\documentclass[12pt,reqno]{amsart}
\usepackage[utf8]{inputenc}
\usepackage{amsfonts}
\usepackage{latexsym}
\usepackage{amssymb}
\usepackage{amsmath}
\usepackage{color}
\usepackage{bbm}
\usepackage{tikz}
\usepackage{enumerate}
\usepackage{mathrsfs}
\usepackage{todonotes}
\usepackage{hyperref}

\usepackage[left=1.2cm, top=2.5cm,bottom=2.5cm,right=1.2cm]{geometry}
\definecolor{col1}{rgb}{0.3.0.4,0.6}

\newcommand{\R}{\mathbb R}
\newcommand{\N}{\mathbb N}
\newcommand{\C}{\mathbb C}
\newcommand{\E}{\mathbb E}

\def\dint{\textup{d}}

\newtheorem{thm}{Theorem}[section]

\newtheorem{cor}[thm]{Corollary}
\newtheorem{lemma}[thm]{Lemma}
\newtheorem{df}[thm]{Definition}
\newtheorem{proposition}[thm]{Proposition}

\newtheorem{thmalpha}{Theorem}

{
\theoremstyle{definition}

}

\definecolor{new}{rgb}{0.1.0.4,0.5}

\definecolor{gurot}{RGB}{180,20,20}
\definecolor{jorot}{RGB}{220,20,20}

\usepackage{nomencl}
\makenomenclature

\usepackage{mathtools}
\mathtoolsset{showonlyrefs}
\usepackage{hyperref}
\hypersetup{
urlcolor=black, 
  menucolor=black, 
  citecolor=black, 
  anchorcolor=black, 
  filecolor=black, 
  linkcolor=black, 
  colorlinks=true,
}
\usepackage{dsfont}
\newcommand{\1}{\mathds{1}}
\newcommand{\ind}{\mathds{1}}

\definecolor{darkblue}{rgb}{.1, 0.1,.8}
\definecolor{darkgreen}{rgb}{0,0.8,0.2}
\definecolor{darkred}{rgb}{.8, .1,.1}

\newcommand{\e}{\operatorname{e}}
\renewcommand{\P}{\mathbbm{P}}
\newcommand{\vep}{\varepsilon}
\newcommand{\cid}{\stackrel{\rm d}{\rightarrow}}
\newcommand{\cip}{\stackrel{\P}{\rightarrow}}
\newcommand{\eid}{\stackrel{\rm d}{=}}

\newcommand{\twonorm}[1]{\|#1\|_2}

\newcommand{\nto}{n \to \infty}

\newcommand{\rhs}{right-hand side}

\newcommand{\Var}{\operatorname{Var}}

\newcommand{\Vol}{\operatorname{Vol}}
\newcommand{\BE}{Berry-Esseen}
\newcommand{\sign}{\operatorname{sign}}
\newcommand{\wt }{\widetilde}
\newcommand{\betabar}{\bar \beta}
\newcommand{\betatilde}{\wt \beta}

\newcommand{\Rtilde}{\wt R}

\def\Y{\mathbf{Y}}
\def\d{\mathrm{d}}
\def\dKS{\mathrm{d}_{KS}}

\begin{document}


\title[]{Thin-shell theory for rotationally invariant random simplices}

\author[J. Heiny]{Johannes Heiny}
\address{Faculty of Mathematics, Ruhr University Bochum, Germany} \email{johannes.heiny@rub.de}

\author[S. Johnston]{Samuel Johnston}
\address{Department of Mathematical Sciences, University of Bath, United Kingdom} \email{sgj22@bath.ac.uk}

\author[J. Prochno]{Joscha Prochno}
\address{Institut f\"ur Mathematik \& Wissenschaftliches Rechnen, Karl-Franzens-Universit\"at Graz, Austria} \email{joscha.prochno@uni-graz.at}

\keywords{Random simplex, logarithmic volume, central limit theorem, high dimension, stochastic geometry, random matrix}
\subjclass[2010]{Primary: 60F05, 52A23 Secondary: 60D05, 60B20}

\thanks{JH was supported by the Deutsche Forschungsgemeinschaft (DFG) via RTG 2131 High-dimensional Phenomena in Probability -- Fluctuations and Discontinuity. 
SJ and JP have been supported by the Austrian Science Fund (FWF) Project P32405 “Asymptotic Geometric Analysis and Applications” of which JP is principal investigator. JP has also been supported by the FWF Project F5508-N26, which is part of the Special Research Program `Quasi-Monte Carlo Methods: Theory and Applications'.
}


\begin{abstract} 
For fixed functions $G,H:[0,\infty)\to[0,\infty)$, consider the rotationally invariant probability density on $\mathbb{R}^n$ of the form 
\begin{align*}
\mu^n( \mathrm{d} s) = \frac{1}{Z_n} G(\twonorm{s})\, \e^{ - n H( \twonorm{s})} \mathrm{d}s.
\end{align*}
We show that when $n$ is large, the Euclidean norm $\twonorm{Y^n}$ of a random vector $Y^n$ distributed according to $\mu^n$ satisfies a thin-shell property, in that its distribution is highly likely to concentrate around a value $s_0$ minimizing a certain variational problem. Moreover, we show that the fluctuations of this modulus away from $s_0$ have the order $1/\sqrt{n}$ and are approximately Gaussian when $n$ is large.

We apply these observations to rotationally invariant random simplices: the simplex whose vertices consist of the origin as well as independent random vectors $Y_1^n,\ldots,Y_p^n$ distributed according to $\mu^n$, ultimately showing that the logarithmic volume of the resulting simplex exhibits highly Gaussian behavior. 
Our class of measures includes the Gaussian distribution, the beta distribution and the beta prime distribution on $\mathbb{R}^n$, provided a generalizing and unifying setting for the objects considered in Grote-Kabluchko-Th\"ale [Limit theorems for random simplices in high dimensions, \emph{ALEA, Lat. Am. J. Probab. Math. Stat.} \textbf{16}, 141–177 (2019)].

Finally, the volumes of random simplices may be related to the determinants of random matrices, and we use our methods with this correspondence to show that if $A^n$ is an $n \times n$ random matrix whose entries are independent standard Gaussian random variables, then there are explicit constants $c_0,c_1\in(0,\infty)$ and an absolute constant $C\in(0,\infty)$ such that
\begin{align*}
\sup_{ s \in \mathbb{R}} \left| \mathbb{P} \left[ \frac{ \log \mathrm{det}(A^n) - \log(n-1)! - c_0  }{ \sqrt{ \frac{1}{2} \log n + c_1  }} < s \right] - \int_{-\infty}^s \frac{ \e^{ - u^2/2} \mathrm{d} u}{ \sqrt{ 2 \pi }} \right| < \frac{C}{\log^{3/2} n},
\end{align*}
sharpening the $1/\log^{1/3 + o(1)}n$ bound in Nguyen and Vu [Random matrices: Law of the determinant, Ann. Probab. 42(1) (2014), 146--167]. 
\end{abstract}

\maketitle


\section{Introduction}

\subsection{High-dimensional probability and random simplices} 
High-dimensional probability theory is concerned with random objects, their characteristics, and the phenomena that accompany both as the dimension of the ambient space tends to infinity. It is a flourishing area of mathematics not least because of numerous applications in modern statistics and machine learning related to high-dimensional data, for instance, in form of dimensionality reduction \cite{BM2001}, clustering \cite{M2018}, principal component regression \cite{S2018}, community detection in networks \cite{FH2016, LLV2018}, topic discovery \cite{DRIS2013}, or covariance estimation \cite{CRZ2016, V2018}. High-dimensional probability bears strong connections to geometric functional analysis and convex geometry and this propinquity is typically reflected both in the flavor of a result and the methods used to obtain it. One of the early results of the theory is commonly known as the Poincar\'e-Maxwell-Borel Lemma (see \cite{DiaconisFreedman}) and states that the first $k$ coordinates of a point uniformly distributed over the $n$-dimensional Euclidean ball (or sphere) of radius $\sqrt{n}$ are independent standard normal variables in the limit as $n\to\infty$ with $k$ fixed. 
Ever since, a variety of limit theorems has been obtained, many of those with the purpose to understand the geometry of high-dimensional convex bodies. Among others, there is Schmuckenschl\"ager's central limit theorem related to the volume of intersections of $\ell_p^n$-balls \cite{Schmu2001} and its multivariate version by Kabluchko, Prochno, and Th\"ale who also obtained moderate and large deviations principles \cite{KPT2019,KPT2020}. Then there is the prominent central limit theorem for convex bodies proved by Klartag, showing that most lower-dimensional marginals of a random vector uniformly distributed in an isotropic convex body are approximately Gaussian \cite{KlartagCLT}, and a number of other results in which limit theorems related to analytic and geometric aspects of high-dimensional objects have been established \cite{APT2018,APTCLT,APT2020,ABP2003, BV2007,BT2019,BobkovKoldobsky,GKR2017,HR2005,JP2019,KPT2019_cube,KPT2020_Sanov,KR2019_thinshell,KimRamanan:2018,MeckesM12,R2005,SS1991,Stam82,T2018, ThaeleTurchiWespi}. 

\subsection{Rotationally invariant random simplices}
The focus of the current paper is rotationally invariant random simplices. Suppose $p,n\in\N$ with $1 \leq p \leq n$ and that $y_1,\ldots,y_p$ are vectors in $\mathbb{R}^n$, and consider the simplex
\begin{align} \label{def:simplex}
\Delta(y_1,\ldots,y_p) := \bigg\{ \sum_{i=1}^ps_i y_i \,:\, s_1,\dots,s_p \geq 0 \quad\text{and}\quad \sum_{i = 1}^p s_i \leq 1 \bigg\},
\end{align}
whose vertices are given by $\{0,y_1,\ldots,y_p \}$. Whenever $p\le n$ and the vectors $y_1,\ldots,y_p$ are linearly independent, this simplex is a $p$-dimensional convex body within $n$ dimensional Euclidean space with non-zero $p$-volume, and this volume may be written in terms of the representation
\begin{align} \label{eq:simplex}
\mathrm{Vol}_p \left( \Delta(y_1,\ldots,y_p)  \right) = \frac{1}{p!} \sqrt{\det_{i,j = 1}^p \langle y_i, y_j \rangle},
\end{align}
where $\langle \cdot, \cdot \rangle$ is the standard Euclidean inner product on $\mathbb{R}^n$ and $\mathrm{Vol}_p$ the $p$-dimensional Lebesgue measure. The primary focus of this paper is in the study of the asymptotics of the  random variable
\begin{align*}
W_{n,p} := \log \mathrm{\Vol}_p \left( \Delta( Y_1^n,\ldots,Y^n_p ) \right)
\end{align*}
given by the logarithmic volume of a simplex whose vertices $Y_1^n,\ldots,Y_p^n$ are independent random vectors in $\mathbb{R}^n$. 

Before proceeding, let us mention here that various related models for random simplices have been considered in the literature. Recently Akinwande and Reitzner obtained multivariate central limit theorems for random simplicial complexes \cite{AR2019}, Gusakova and Th\"ale studied the logarithmic volume of simplices in high-dimensional Poisson-Delaunay tessellations and obtained several types of limit theorems \cite{GT2019}, and Grote, Kabluchko, and Th\"ale \cite{GKT2019} investigated the logarithmic volume for other classes of random simplices such as those generated by Gaussian, Beta or the spherical distribution. We should remark here that with a view to drawing on connections with random matrices, we study random simplices for which the origin is a fixed vertex, where as in the chief focus of Grote, Kabluchko, and Th\"ale are random simplices all of whose vertices are random. A central limit theorem for random simplices arising from product distributions with sub-exponential tails was treated by Alonso-Guti\'errez et al. in \cite{Aetal2019}.

In view of the recent works \cite{Aetal2019} and \cite{GKT2019}, we work more generally, making the sole restriction that the law of the simplex is invariant under rotations of the underlying space, which occurs whenever the vectors $Y_1^n,\ldots,Y_p^n$ are drawn independently according to a probability distribution $\mu$ on $\mathbb{R}^n$ that is rotationally invariant, in the sense that
\begin{align*}
\mu( T (A)) = \mu(A)
\end{align*}
for every Borel subset $A$ of $\mathbb{R}^n$ and every linear orthogonal transformation $T: \mathbb{R}^n \to \mathbb{R}^n$. Given that a random variable $Y$ is distributed according to a rotationally invariant probability distribution, we can decompose $Y$ so that
\begin{align}\label{eq:radial decoupling}
Y \eid R \Theta,
\end{align}
where $R\eid \|Y\|_2$ is a $[0,\infty)$-valued random variable independent of the random vector $\Theta$ uniformly distributed on the Euclidean unit sphere $\mathbb{S}^{n-1}:=\big\{ (x_i)_{i=1}^n\,:\, \sum_{i=1}^nx_i^2=1 \big\}$. Here and elsewhere $\eid$ refers to equality in distribution.
We would like to emphasize that this framework encompasses the spherical, Gaussian, beta and beta prime models considered in \cite{GKT2019}, where in these natural contexts as in many others the distribution of the radial part $R$ tends to vary with the underlying dimension $n$. 
 
The radial decoupling \eqref{eq:radial decoupling} behaves agreeably with the determinant, in that if $\left\{ Y_i = R_i \Theta_i : i = 1,\ldots, p \right\}$ are independent and identically distributed decompositions of the law $\mu$, by \eqref{eq:simplex} we have the decoupling of the volume
\begin{align} \label{eq:decoupling}
\Vol_p \big( \Delta(Y_1,\ldots,Y_p) \big) &= \frac{1}{p!} \sqrt{ \det_{i,j = 1}^p \langle Y_i, Y_j \rangle} \nonumber \\
&= \frac{1}{p!} \left( \det_{i,j = 1}^p \left( \twonorm{Y_i} \twonorm{Y_j}  \left\langle \frac{Y_i}{\twonorm{Y_i}}, \frac{Y_j}{\twonorm{Y_j}} \right\rangle \right) \right)^{1/2} \nonumber \\
&\stackrel{\dint}{=} \frac{1}{p!} \left( \prod_{ k =1}^p R_k \right) \sqrt{ \det_{i,j = 1}^p \langle \Theta_i, \Theta_j \rangle}.
\end{align}
In particular, there are two independent sources of variance that contribute to the simplicial volume: the product $\prod_{ k =1}^p R_k$ and the spherical determinant $\left( \det_{i,j = 1}^p \langle \Theta_i, \Theta_j \rangle \right)^{1/2}$. Before going any further, we take a moment to focus on this latter term, which we would obtain in \eqref{eq:decoupling} if $\mu$ was the uniform distribution on the unit sphere $\mathbb{S}^{n-1}$ in $\mathbb{R}^n$, so that in the above decomposition each $R_i$ would be equal to $1$ almost surely. In this case it was first observed by Miles \cite{M1971} 
that we have the distributional identity 
\begin{align} \label{eq:product rep}
 \det_{i,j = 1}^p \langle \Theta_i, \Theta_j \rangle \stackrel{\dint}{=} \prod_{ j = 1}^{p-1} \beta_{\frac{n-j}{2},\frac{j}{2} },
\end{align}
where the terms in the product on the right-hand side are independent random variables such that each $\beta_{(n-j)/2,j/2}$ is beta distributed with shape parameters $(n-j/2,j/2)$. 

Recall now that a random variable follows a beta distribution with shape parameters $\alpha,\beta>0$ if it has Lebesgue density on $[0,1]$ given by
\[
x\mapsto \frac{\Gamma(\alpha+\beta)}{\Gamma(\alpha)\Gamma(\beta)}x^{\alpha-1}(1-x)^{\beta-1}.
\]
In view of \eqref{eq:decoupling} and \eqref{eq:product rep}, the logarithmic volume thus satisfies 
$$
\log \mathrm{Vol}_p \big( \Delta(\Theta^n_1,\ldots,\Theta_p^n)\big) \stackrel{\dint}{=}-\log p! + \frac{1}{2}\sum_{j=1}^{p-1} \log \beta_{\frac{n-j}{2},\frac{j}{2} }. 
$$
This representation of the log-volume of a spherical random simplex as a sum of independent random variables is utilized by Grote et al. \cite{GKT2019} to obtain a Berry-Esseen bound 
\begin{align*}
\mathrm{d}_{KS} \left( \frac{ W_{n,p} - \mathbb{E}[ W_{n,p} ]}{ \sqrt{ \mathrm{Var}[ W_{n,p} ] } }  , N \right)   \leq \frac{C}{\sqrt{p}}
\end{align*} 
for the logarithmic volume $W_{n,p}$ of a random simplex with spherically distributed vertices, where
\begin{align*}
\mathrm{d}_{KS} \left( X  , N \right) = \sup_{ r \in \mathbb{R} } \left| \mathbb{P} \left[ X < r \right] - \int_{-\infty}^r \frac{ \e^{ - u^2/ 2 } \mathrm{d} u }{ \sqrt{2 \pi}} \right|  
\end{align*}
is the Kolmogorov-Smirnov distance between a random variable  $X$ and $N$. {\it Throughout this paper, $N$ denotes a standard Gaussian random variable.} 

In fact, Grote et al. \cite{GKT2019} find representations analogous to \eqref{eq:product rep} suitable for simplices with Gaussian, beta, and beta prime distributed vertices, and prove analogous Berry-Esseen bounds in these settings.

\subsection{An overview of our results} 

In this paper we work in a broad setting, considering random simplices whose vertices are random vectors distributed according to one of a large class of rotationally invariant probability distributions. We find that the volumes of these simplices exhibit a certain interplay of high dimensional phenomena creating what might be described as \emph{extremely Gaussian} behavior. With a view to outlining these phenomena here, by combining the decoupling representation \eqref{eq:decoupling} with the spherical identity \eqref{eq:product rep} and taking logarithms, we obtain the distributional identity
\begin{align} \label{eq:main rep}
W_{n,p} := \log \mathrm{Vol}_p \left( \Delta(Y_1^n,\ldots,Y_p^n ) \right) \stackrel{\dint}{=} - \log p! + \frac{1}{2} \sum_{ j  = 1}^{ p - 1} \log \beta_{ \frac{n-j}{2} , \frac{j}{2} } + \sum_{ j = 1}^p \log R^n_j ,
\end{align}
constituting the starting point for our analysis. Since $W_{n,p}$ has a representation of order $p$ independent random variables, it is natural to expect, provided the moments are sufficiently regular, that when appropriately normalized, $W_{n,p}$ converges at a speed $1/ \sqrt{p}$ in distribution to a standard Gaussian random variable; see Figure \ref{fig:1} for an illustration. 

\begin{figure}[htb!]
  \centering
    \includegraphics[trim = 0in 0in 0in 0in, clip, scale=0.35]{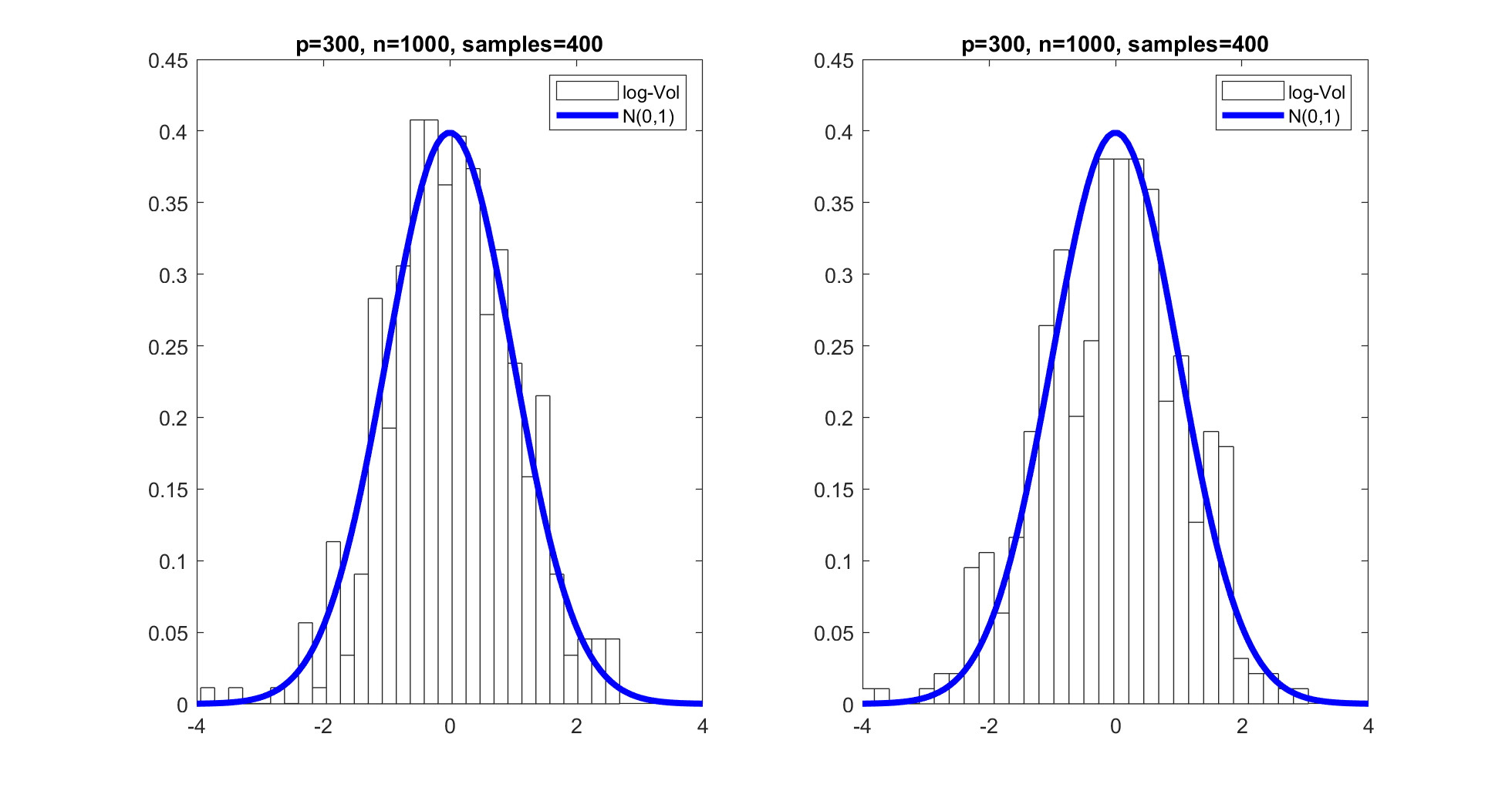}
\caption{Histograms of $400$ simulated $\log \mathrm{Vol}_p \left( \Delta( Y_1^n ,\ldots, Y_p^n ) \right)$, properly centered and standardized (as in \cite{heiny:parolya:kurowicka:2021}), for $p=300$ and $n=1000$. Left: $Y_i^n$ uniformly distributed on the unit sphere. Right: $Y_i^n$ with i.i.d.~standard normal components.}
\label{fig:1}
\end{figure}

We find that in a host of natural settings (including Gaussian, beta and beta prime simplices) the dimension $n$ of the ambient space also contributes to creating Gaussian behavior at a speed much faster than the $1/ \sqrt{p}$ speed predicted by the Berry-Esseen theorem. We prove this by way of a pincer strategy, handling differently the distinct sums on the right-hand side of \eqref{eq:main rep}. More concretely, we will focus on random simplices in which the random vectors $Y_1^n,\ldots,Y_p^n$ are distributed according to a probability measure of the form 
\begin{align} \label{eq:rotform}
\mu^n(\d s) = \frac{1}{Z^n} G( \twonorm{s} ) \exp \left( - n H( \twonorm{s}\right) \mathrm{d} s,
\end{align}
where $\twonorm{s} := \left( \sum_{i=1}^n s_i^2 \right)^{1/2}$ is the $\ell_2$-norm, $G,H:[0,\infty) \to [0,\infty)$ are functions satisfying some mild conditions and $Z^n$ is a normalization constant.  
This class of probability measures includes the beta and beta prime models, as well as the Gaussian model (which is obtained after rescaling by $1/\sqrt{n}$).
Let us briefly outline our results here in the case where $p \leq \theta n$ for a fixed $\theta < 1$:
\begin{itemize}
\item Consider the random variable
\begin{align}\label{eq:dddf}
W_{n,p}^{ \mathrm{Sph} } := \log \mathrm{Vol}_p \left( \Delta( \Theta_1^n , \ldots, \Theta^n_p) \right) \stackrel{\dint}{=} -\log p! +\frac{1}{2} \sum_{ j  = 1}^{ p - 1} \log \beta_{ \frac{n-j}{2} , \frac{j}{2} }.
\end{align}
Using a Fourier-analytic approach, we prove a fast Berry-Esseen bound for the random variable $W_{n,p}^{\mathrm{Sph}}$, the word `fast' being used here to indicate that the speed of the bound exceeds $1/\sqrt{p}$ for $p$ uniformly bounded away from $n$.

\item The next major step in our work is a thin-shell type result for rotationally invariant probability distributions of the form \eqref{eq:rotform}. Namely, we show that if $X^1,X^2,X^3\ldots$ is a sequence of independent random vectors such that each $X^i$ takes values in $\mathbb{R}^i$ and is distributed according to $\mu^i$, then the sequence of their standardized log-radii
\begin{align*}
\tilde{R}^i := \frac{  \log \twonorm{X^i} - \mathbb{E}[ \log \twonorm{X^i } ] }{ \sqrt{ \mathrm{Var}[\log \twonorm{X^i}] }},\quad i\in\N,
\end{align*}
converges in distribution to a standard Gaussian random variable as $i\to\infty$. Theorem \ref{thm:laplace berry} below says something stronger however: there is a constant $L = L(G,H)\in(0,\infty)$ depending on the functions $G$ and $H$ but independent of $p$ and $n$ such that if $X_1^n,\ldots,X^n_p$ are independent and identically distributed according to $\mu^n$, and $\tilde{R}^n_1,\ldots,\tilde{R}^n_p$ are their associated normalised log-radii, then we have the fast Berry-Esseen bound
\begin{align*}
\dKS \left(\frac{ \tilde{R}^n_1 + \ldots + \tilde{R}^n_p}{ \sqrt{p} }\, ,\, N \right)\leq L \left( \frac{1}{ \sqrt{pn }} + \e^{ - c p} \right),
\end{align*}
where $c>0$ is an absolute constant.
\item The two prior results state that both $\sum_{ j = 1}^p \log R_j^n$ and $\sum_{ j = 1}^{ p- 1} \log \beta_{ \frac{n-j}{2} , \frac{j}{2} }$ are both within certain Kolmogorov-Smirnov distances of Gaussian random variables with certain means and variances. These results may be combined fairly quickly using a triangle inequality for Kolmogorov-Smirnov distances, leading to a proof of Theorem \ref{thm:main}.
\end{itemize}

Finally, we drop assumption \eqref{eq:rotform}, which is used to relate the distribution of the log-radius of the vectors $Y_i^n$ to a Gaussian distribution, and consider more general rotationally invariant random vectors $Y_i^n$. In particular, the log-radius might be in the domain of attraction of an infinite variance stable distribution. Depending on the relation between the tails of the log-radius of $Y_i^n$ and the variance of $W_{n,p}^{ \mathrm{Sph} }$ we find that the properly normalized $\log \mathrm{Vol}_p \left( \Delta(Y_1^n,\ldots,Y_p^n ) \right)$ either converges to a normal limit, an infinite variance stable limit or a mixture between those two.


\subsection*{Overview of the remainder of the paper}

The rest of the paper is structured as follows. 
In Section \ref{sec:mainresults}, we present our main results. First, we present a fast Berry-Esseen theorem for the log-volume of the spherical simplex which is then extended to rotationally invariant random simplices. As a byproduct, we prove a \BE~type result for the sum of iid random variables whose density resembles the Gaussian density. 

In Section~\ref{sec:maina}, we provide limit theory for the log-volume of the rotationally invariant random simplices under general conditions, also allowing for very heavy-tailed distributions. Section~\ref{sec:rmtperspective} highlights the connection of our findings to random matrix theory. As an application of our results we prove convergence of the logarithmic determinant of an iid standard Gaussian random matrix at speed $(\log n)^{-3/2}$. 

Section \ref{sec:spherical}-\ref{sec:mainproofs} are devoted to the proofs of the results in Section \ref{sec:mainresults}.
In Section \ref{sec:spherical}, we begin with a careful analysis of random simplices whose vertices are $p$ points chosen uniformly on the unit sphere in $\mathbb{S}^{n-1}$, culminating in a proof of Theorem \ref{thm:spherical berry}.
In Section \ref{sec:laplace}, we introduce our probabilistic approach to the Laplace method, ultimately working towards a proof of Theorem \ref{thm:laplace berry}.
Section \ref{sec:main} combines our work in the prior two sections together to prove Theorem \ref{thm:main}.
In the next section, Section \ref{sec:gaussian proof}, we give a short proof of Theorem \ref{thm:gaussian matrix}, using some of the machinery developed in Section \ref{sec:spherical}.
Finally, all results from Section \ref{sec:maina} are proved in Section~\ref{sec:mainproofs}.

\section{Main results} \label{sec:mainresults}

In this section we state our results in full. 

\subsection{A fast Berry-Esseen theorem for the log-volume of the spherical simplex}
Our first result is a Berry-Esseen bound for the spherical random simplex. Here, for integers $p \leq n$ let
\begin{align}\label{eq:Wspherical}
\widetilde{W}_{n,p}^{\mathrm{Sph}} := \frac{ W_{n,p}^{\mathrm{Sph}} - \mathbb{E} [ W_{n,p}^{\mathrm{Sph}} ] }{ \sqrt{ \mathrm{Var}[ W_{n,p}^{\mathrm{Sph}} ] }}
\end{align}
denote the standardized log-volume of the spherical random simplex associated with $p$ points chosen independently and uniformly at random from $\mathbb{S}^{n-1}$.

\begin{thmalpha} \label{thm:spherical berry}
Let $p,n\in\N$ such that $p\leq n$ and $\widetilde{W}_{n,p}^{\mathrm{Sph}}$ be the normalized log-volume of a spherical simplex. Then there is a universal constant $C\in(0,\infty)$ such that whenever $p \geq 41$,
\begin{align*}
\dKS \left( \widetilde{W}_{n,p}^{\mathrm{Sph}},N\right)  \leq   \frac{ C ~ \theta^2  }{ n(1-\theta) \left[ \log \frac{1}{1-\theta} - \theta \right]^{3/2}  }, 
\end{align*}
where $\theta := \theta(p,n):= \frac{p-1}{n}$. In fact, we may take $C = 28$. 
\end{thmalpha}
We take a moment to unpack the bound in Theorem \ref{thm:spherical berry} by looking at the following easily verified consequences:
\begin{itemize}
\item Fix $\phi \in (0,1)$. Using the inequality $\log \frac{1}{1-\theta} - \theta \geq \theta^2/2$ for $\theta\in [0,1)$, it is easily verified that whenever $\frac{p-1}{n} = \theta \leq \phi$, we have
\begin{align}\label{eq:sedsdd}
\dKS \left( \widetilde{W}_{n,p}^{\mathrm{Sph}},N\right)  \leq \frac{ C_\phi}{ p-1},
\end{align}
where $C_\phi := 2 \sqrt{2}  C / (1 - \phi)$. 
\item On the other hand, for all $p \leq n$ by setting $q := n- p +1$ (so that $q=n(1-\theta)$), we have
\begin{align} \label{eq:eisbar}
\dKS \left( \widetilde{W}_{n,p}^{\mathrm{Sph}},N\right)  \leq \frac{C}{ q ( \log (n/q) - 1 )^{3/2} }.
\end{align}
\end{itemize}

Let us remark here that an analogous result to Theorem \ref{thm:spherical berry} appears in Section 3 of Grote et al. \cite[Theorem 3.6]{GKT2019}, who in contrast to us consider random simplices \emph{not} having the origin as a fixed vertex. They obtain a similar bound in the case where $\frac{p-1}{n}$ is bounded away from one, though their bound is weaker in the $n - p = o(n)$ case; they obtain $C/\log^{1/2}(n/q)$ in the setting of \eqref{eq:eisbar}.

\subsection{A Berry-Esseen theorem for the Laplace method}

Our next result concerns the highly Gaussian behavior of the sums of the log-radii. Here we take a moment to give a brief digression on the \emph{Laplace method}, which states that when $g$ and $h$ are suitably regular functions with $h$ attaining a global minimum at some $x_0 \in (a,b)$, then we have the asymptotics 
\begin{align} \label{eq:laplace}
\int_a^b g(x) \e^{ - n h(x) } \mathrm{d} x = (1 + o (1)) \sqrt{ \frac{ 2 \pi}{ n h''(x_0) }} g(x_0) \exp( - n h(x_0) )
\end{align}
as $n \to \infty$. See e.g. \cite{AGZ}. The key conceptual point in the Laplace method is that, thanks to the Taylor expansion $n h(x_0 + w/\sqrt{n}) \approx w^2 h''(x_0)/2 + O( w^3/\sqrt{n})$, the integral in \eqref{eq:laplace} behaves roughly like a Gaussian integral around $x_0$. 

Theorem \ref{thm:laplace berry} develops this idea further, stating that when $n$ is large, random variables whose probability distributions take the form \eqref{eq:laplace} are approximately Gaussian. To set this up, we require some conditions on the functions. For a fixed pair $(g,h)$, we consider probability density functions of the form 
\begin{align} \label{eq:rho def}
\rho^n(\d x)  = \frac{1}{Z^n} g(x) \exp \left( - n h(x) \right) \mathrm{d} x,\qquad n\in\N,
\end{align}
where the $Z^n\in(0,\infty)$ are normalization constants, and the ordered pair of functions $g,h : [0,\infty) \to [0,\infty)$ is \emph{admissible} per the following definition.

\begin{df} \label{df:admissible}
Let $g,h:\mathbb{R} \to [0,\infty)$ be two functions. We say the pair $(g,h)$ is admissible if and only if (a)-(c) hold.
\begin{itemize}
\item[(a)] The density function $\rho^n$ is differentiable almost-everywhere, and has a unique maximum at a point $x_0$ in $\mathbb{R}$ such that $x_0$ is a minimum of $h$. Moreover, we assume that $\rho^n$ is increasing on $(-\infty,x_0]$ and decreasing on $[x_0,\infty)$. 
\item[(b)] In a neighbourhood $[x_0 - \delta, x_0 + \delta]$ of $x_0$, $h$ is twice differentiable. Moreover, if we write
\begin{align*}
h(x) = h(x_0) + h''(x_0) \left( \frac{1}{2} (x-x_0)^2 + r(x) \right) \qquad \text{and} \qquad g(x) = g(x_0) (1 + q(x-x_0)),
\end{align*}
then we have
\begin{align*}
|r(x) | \leq \frac{1}{4\delta} |x-x_0|^3 \qquad \text{and} \qquad |q(x)| \leq \frac{1}{4 \delta} |x-x_0|.
\end{align*}
\item[(c)] Outside of this neighborhood, i.e., for each $x\in\R\setminus [x_0 - \delta, x_0 + \delta]$, there exist constants $\alpha,c,C \in (0,\infty)$ such that  
\begin{align*}
h(x)  \geq  c \log (1 + |x-x_0|) \qquad \text{and} \qquad g(x) \leq C ( 1 + |x-x_0|^\alpha ).
\end{align*}
\end{itemize}
\end{df}
As an immediate consequence of Theorem \ref{thm:laplace berry} below, all probability distributions with densities of the form \eqref{eq:rho def} that satisfy Definition~\ref{df:admissible} are in the domain of attraction of the normal law. In particular, they include the Gaussian distribution, the Gamma distribution and the beta distribution.

Our Berry-Esseen theorem for the Laplace method states that when $n$ is large, the normalized sum of $p$ independent random variables distributed according to an admissible density $\rho^n$ is close in distribution to a standard Gaussian random variable $N$. 

\begin{thmalpha} \label{thm:laplace berry}
For an admissible pair $(g,h)$ there is a constant $C_{g,h}\in(0,\infty)$ and $n_0\in\N$ such that for all $n \geq n_0$ we have the following: if $X_1^n,\ldots,X^n_p$ are independent random variables with density $\rho^n$ given by \eqref{eq:rho def}, then 
\begin{align*}
\mathrm{d}_{KS} \left( \frac{  \sum_{ i = 1}^p X^n_i - p \mathbb{E}[X_1^n] }{ \sqrt{ p \,\mathrm{Var}[X_1^n] } } , N \right) \leq C_{g,h} \left( \frac{1}{ \sqrt{pn } } + 2^{ - p} \right).
\end{align*}
\end{thmalpha}

The value of Theorem \ref{thm:laplace berry} lies in its application to a sort of thin-shell property for a large class of radial densities on $\mathbb{R}^n$. To this end, we say a pair $(G,H)$ of functions $G,H:[0,\infty) \to \mathbb{R}$ are \emph{radially admissible} if the pair $(g,h)$ given by 
\begin{align*}
g(s) := G(\e^s) \qquad \text{and} \qquad h(s) := H(\e^s) - s
\end{align*}
are admissible.

Suppose now for a fixed radially admissible pair $(G,H)$ for each $n\in\N$ we have a rotationally invariant probability density on $\mathbb{R}^n$ of the form
\begin{align} \label{eq:mu def}
\mu^n(s) := \frac{1}{Z^n} G( \twonorm{s} ) \exp \left( - n H( \twonorm{s} ) \right),\qquad s\in \R^n,
\end{align}
where $Z^n$ is the normalizing constant. 
Then by virtue of a straightforward calculation involving the polar integration formula, if $X^n$ is a random vector distributed according to $\mu^n$, then its log radius $\log \twonorm{X^n}$ has the density 
\begin{align*}
\frac{1}{ \widetilde{Z}^n} G(\e^r) \exp\left( - n (H(\e^r) - r) \right) = \frac{1}{ \widetilde{Z}^n}\, g(r) \e^{ - nh(r) }, \qquad r\in \R,
\end{align*}
on the real line, where $\widetilde{Z}^n$ is again a normalizing constant. This observation is one of the key ingredients in synthesizing Theorem \ref{thm:laplace berry} with Theorem \ref{thm:spherical berry} to obtain the following general result.

\begin{thmalpha} \label{thm:main}
For each radially admissible pair $G,H$ there is a constant $C_{G,H}\in(0,\infty)$ and $n_0\in\N$ such that for all $n \geq n_0$ we have the following: if $Y_1^n,\ldots,Y_p^n$ are $p$ independent random vectors in $\mathbb{R}^n$ distributed according to $\mu^n$ as it appears in \eqref{eq:mu def}, 
then for $W^{G,H}_{n,p} := \log \mathrm{Vol}_p \left( \Delta( Y_1^n ,\ldots, Y_p^n ) \right)$ it holds that
\begin{align*}
\mathrm{d}_{KS} \left( \frac{W^{G,H}_{n,p} - \mathbb{E}[W^{G,H}_{n,p}]  }{ \sqrt{ \mathrm{Var}[W^{G,H}_{n,p}]}  } , N \right) \leq C_{G,H} \frac{  \theta^2  }{ n(1-\theta) \left[ \log \frac{1}{1-\theta} - \theta \right]^{3/2}  },
\end{align*}
where $\theta := \theta(p,n):=\frac{p-1}{n}$. 
\end{thmalpha}

We take a moment to highlight two special cases of Theorem \ref{thm:main}. 
\begin{itemize}
\item The case where $G(x) = 1$ is the identity map and $H(x) = x^2/2$ corresponds to the Gaussian distribution with covariance matrix $\frac{1}{n} I_n$, where $I_n$ denotes the $n\times n$ identity matrix.
\item The case where $G(x) = 1$ is the identity map and $H(x) = \frac{1 + \phi}{2} \log (1 + |x|^2 )$ corresponds to the so called \emph{Beta prime} distribution on $\mathbb{R}^n$ with parameter $\nu = \phi n$, where $\phi > 0$.
\end{itemize}

\subsection{Fluctuations of the log-volume under general conditions}\label{sec:maina}

In this subsection, we work more generally and drop assumption \eqref{eq:rho def}, which was used to relate the distribution of the log-radius of the vectors $Y_i^n$ to a Gaussian distribution. We consider iid, rotationally invariant random vectors $Y_i^n$, which we collect in the data matrix 
\begin{equation}\label{eq:datamatrix}
\Y:=\Y_n:=(Y_1^n,\ldots,Y_p^n)\,.
\end{equation}

The main focus is no longer on deriving fast \BE~bounds for the convergence to the Gaussian distribution. Our goal is to study the asymptotic distribution of $\log \Vol_p \left( \Delta\Y  \right)$ for a wide range of radial laws. 
 For the number of points constituting our simplex,
we  consider the asymptotic regime
  \begin{equation}\label{Cgamma}
    p=p_n \to \infty \quad \text{ and } \quad p\le n\,,\quad \text{ as } \nto\,.
  \end{equation}

To simplify notation,  we define the random variable $R_{(n)}=\twonorm{Y_1^n}$ and set $\wt R_{(n)} =\log R_{(n)}$. 
For the field $(\beta_{i/2,j/2})_{i,j\in \N}$ of independent random variables such that $\beta_{i/2,j/2}$ is $\text{Beta}(i/2,j/2)$ distributed, we write $\wt \beta_{i/2,j/2}=\log \beta_{i/2,j/2}$.

Our next result provides conditions on the radius $R_{(n)}$ under which the fluctuations of the log-volume about its mean are asymptotically Gaussian.

\begin{thmalpha}[Normal limit]\label{thm:normalmain}
Under the growth condition \eqref{Cgamma}, consider the data matrix $\Y$ defined in \eqref{eq:datamatrix} with independent and rotationally invariant columns, i.e \eqref{eq:radial decoupling} holds. Assume there exists a sequence of positive constants $\sigma_n$ such that, as $\nto$, 
\begin{equation}\label{eq:condnormal1}
\begin{split}
&\frac{p\left( \E[\wt R_{(n)}^2 \1_{\{|\wt R_{(n)}|<\sigma_n\}}]-(\E[\wt R_{(n)} \1_{\{|\wt R_{(n)}|<\sigma_n\}}])^2 \right)}{\sigma_n^2}\\
&\quad +\frac{1}{4\sigma_n^2}\sum_{j=1}^{p-1} \left( \E[\betatilde_{ \frac{n-j}{2} , \frac{j}{2} }^2 \1_{\{|\betatilde_{ \frac{n-j}{2} , \frac{j}{2} }|<2\sigma_n\}}]-(\E[\betatilde_{ \frac{n-j}{2} , \frac{j}{2} } \1_{\{|\betatilde_{ \frac{n-j}{2} , \frac{j}{2} }|<2\sigma_n\}}])^2 \right) \to 1\,,
\end{split}
\end{equation}
and 
\begin{equation}\label{eq:condnormal2}
p\, \P(|R_{(n)}|\ge \vep  \sigma_n) +\sum_{j=1}^{p-1} \P(|\betatilde_{ \frac{n-j}{2} , \frac{j}{2} }|\ge 2 \vep \sigma_n)\to 0\,, \qquad  \vep>0\,.
\end{equation}
Let $(b_n)_{n\ge 1}$ be a sequence satisfying, as $\nto$, 
\begin{equation}\label{eq:bnormal}
b_n=p\,\E[\wt R_{(n)} \1_{\{|\wt R_{(n)}|<\sigma_n\}}] -\log (p!)+\frac{1}{2} \sum_{j=1}^{p-1}  \E[\betatilde_{ \frac{n-j}{2} , \frac{j}{2} } \1_{\{|\betatilde_{ \frac{n-j}{2} , \frac{j}{2} }|<2\sigma_n\}}] +o(\sigma_n)\,.
\end{equation}
Then we have
\begin{equation}\label{eq:cltnormal}
\frac{\log \Vol_p (\Delta  \Y)  - b_n}{\sigma_n} \cid N\,, \qquad \nto\,.
\end{equation}
\end{thmalpha}
Theorem \ref{thm:normalmain} characterises the distributions of radii such that the logarithmic volume satisfies a central limit theorem. In fact, since Petrov's \cite{petrov:1975} \emph{infinite smallness} condition is always satisfied in our model, a slightly stronger result holds under the assumptions of Theorem \ref{thm:normalmain} and a sequence of positive constants $\sigma_n$. Namely, the existence of a non-random sequence $(b_n)$ such that \eqref{eq:cltnormal} holds  is equivalent to $(\sigma_n)$ satsifying  \eqref{eq:condnormal1} and \eqref{eq:condnormal2}. If \eqref{eq:condnormal1} and \eqref{eq:condnormal2} hold, we may choose $b_n$ as in \eqref{eq:bnormal}.

Our next result shows that the logarithmic volume can also have an $\alpha$-stable limit. In particular, this is the case when $\wt R_{(n)}$ has power law tails with index $\alpha<2$. To the best of our knowledge, the most general setting in which the limiting distribution of the log-volume (or equivalently the log-determinant) was derived was in
\cite{bao2015, wang2018} who assumed that the entries of $\Y$ possess a finite fourth moment, which is the typical assumption in papers on linear spectral statistics. We refer to \cite{auffinger:arous:peche:2009,davis:heiny:mikosch:xie:2016, heiny:mikosch:2017:iid, heiny:mikosch:2017:corr,fleermann:heiny:2020,basrak:heiny:jung:2020,heiny:podolskij:2021,heiny:yao:2021} for collections of results which show the stark differences in the asymptotic behavior under infinite fourth moments. 

In order to present our stable limit theorem, we introduce the auxiliary sequence 
\begin{equation}\label{eq:omega}
\omega_n^2 :=-\frac{1}{2} \log \frac{n-p+1}{n} -\frac{p^2}{2n(p+1)}\,,\qquad n\ge 1\,,
\end{equation}
which one may interpret as the critical variance sequence.
\begin{thmalpha}[$\alpha$-stable limit]\label{thm:stablelimit}
Under the growth condition \eqref{Cgamma}, consider the data matrix $\Y$ defined in \eqref{eq:datamatrix} with independent and rotationally invariant rows, i.e \eqref{eq:radial decoupling} holds.
 For some $\alpha\in (0,2)$ and $c_1,c_2\ge 0$ with $c_1+c_2>0$, assume that there exists a sequence of positive constants $\sigma_n$ such that, as $\nto$, $\omega_n/\sigma_n \to 0$ and
\begin{align}
 &p\, \P(\sigma_n^{-1} \Rtilde_{(n)}\le -x)\to c_1 x^{-\alpha}\,,\quad  p\, \P(\sigma_n^{-1}\Rtilde_{(n)}> x)\to c_2 x^{-\alpha}\,, \qquad x>0\,, \text{ and}\\ 
&\lim_{\vep\to 0} \limsup_{\nto} \frac{p}{\sigma_n^2} \left( \E[\Rtilde_{(n)}^2 \1_{\{|\Rtilde_{(n)}|<\vep \sigma_n\}}]-(\E[\Rtilde_{(n)} \1_{\{|\Rtilde_{(n)}|<\vep \sigma_n\}}])^2 \right) =0\,.
\end{align}
Set $a_{n}=\sigma_n^{-1}\E[\Rtilde_{(n)} \1_{\{|\Rtilde_{(n)}|<\sigma_n\}}]$ and
let $(b_n)_{n\ge 1}$ be a sequence satisfying 
$$b_n=-\log (p!)+\frac{1}{2} \sum_{j=1}^{p-1}  \E[\betatilde_{ \frac{n-j}{2} , \frac{j}{2} } ]+p\sigma_n \left(a_{n} +\int_{-\infty}^{\infty} \frac{x}{1+x^2}\, d\P(\Rtilde_{(n)}\le \sigma_n(x+a_{n})) \right) +o(\sigma_n)\,.$$
Then we have the following weak convergence to an $\alpha$-stable limit:
\begin{equation*}
\frac{\log \Vol_p \left( \Delta \Y \right) - b_n}{\sigma_n} \cid Z_{\alpha}\,, \qquad \nto\,.
\end{equation*}
The limit random variable $Z_{\alpha}=Z_{\alpha}(c_1,c_2)$ has the
characteristic function 
\begin{equation}\label{eq:Zalpha}
\E[\e^{{\rm i} tZ_{\alpha}}]=\left\{\begin{array}{ll}
\exp \left\{  \alpha (c_1+c_2)\Gamma(-\alpha) \cos (\frac{\pi \alpha}{2}) |t|^\alpha \big(1-{\rm i} \eta \tan (\frac{\pi \alpha}{2}) \sign(t)\big) \right\} \,, 
& \mbox{if } \alpha \neq 1, \\
\exp \left\{  - (c_1+c_2)\frac{\pi}{2} |t| \big( 1+{\rm i} \eta \frac{2}{\pi} \sign(t) \log |t| \big) \right\}  \,, 
& \mbox{if } \alpha=1,
\end{array}\right.
\end{equation}
where $\eta=(c_2-c_1)/(c_1+c_2)$.
\end{thmalpha}

Finally, there is an interesting mixed case, when the variances of the two sums on the \rhs~of \eqref{eq:main rep} are of the same order.

\begin{thmalpha}[mixed limit]\label{thm:mixedlimit}
Under the growth condition \eqref{Cgamma}, consider the data matrix $\Y$ defined in \eqref{eq:datamatrix} with independent and rotationally invariant rows, i.e \eqref{eq:radial decoupling} holds.
 For some $\alpha\in (0,2)$ and $c_1,c_2\ge 0$ with $c_1+c_2>0$, assume that there exists a sequence of positive constants $\sigma_n$ such that, as $\nto$, $\omega_n/\sigma_n \to q\in (0,\infty)$ and
\begin{align}
 &p\, \P(\sigma_n^{-1} \Rtilde_{(n)}\le -x)\to c_1 x^{-\alpha}\,,\quad  p\, \P(\sigma_n^{-1}\Rtilde_{(n)}> x)\to c_2 x^{-\alpha}\,, \qquad x>0\,, \text{ and}\\ 
&\lim_{\vep\to 0} \limsup_{\nto} \frac{p}{\sigma_n^2} \left( \E[\Rtilde_{(n)}^2 \1_{\{|\Rtilde_{(n)}|<\vep \sigma_n\}}]-(\E[\Rtilde_{(n)} \1_{\{|\Rtilde_{(n)}|<\vep \sigma_n\}}])^2 \right) =0\,.
\end{align}
Set $a_{n}=\sigma_n^{-1}\E[\Rtilde_{(n)} \1_{\{|\Rtilde_{(n)}|<\sigma_n\}}]$ and
let $(b_n)_{n\ge 1}$ be a sequence satisfying 
\begin{equation*}
\begin{split}
b_n&=-\log (p!)+\frac{1}{2} \sum_{j=1}^{p-1}  \E[\betatilde_{ \frac{n-j}{2} , \frac{j}{2} } ]+p\sigma_n \left(a_{n} +\int_{-\infty}^{\infty} \frac{x}{1+x^2}\, d\P(\Rtilde_{(n)}\le \sigma_n(x+a_{n})) \right) +o(\sigma_n)\,.
\end{split}
\end{equation*}
Then we have 
\begin{equation*}
\frac{\log \Vol_p \left( \Delta \Y \right) - b_n}{\sigma_n} \cid q\,N+Z_{\alpha}\,, \qquad \nto\,,
\end{equation*}
where $Z_{\alpha}=Z_{\alpha}(c_1,c_2)$ has the
characteristic function \eqref{eq:Zalpha}.
\end{thmalpha}

\subsection{The random matrix perspective}\label{sec:rmtperspective}
While we have discussed our results so far from the perspective of the volumes of random simplices, the framework we considered is intimately related to the determinants of random matrices.  Indeed, we saw in \eqref{eq:simplex} that the volume of a simplex with vertices $Y^n_1,\ldots,Y^n_p$ in $\mathbb{R}^n$ may be expressed in terms of a determinant. Developing this equation slightly, we may write
\begin{align} \label{eq:det vol}
\mathrm{Vol}_p \left( \Delta(Y^n_1,\ldots,Y^n_p) \right) =\mathrm{Vol}_p \left( \Delta \Y \right)= \frac{1}{p!} \sqrt{\mathrm{det} (\Y^{\top}\Y)},
\end{align}
where $\Y:=\Y_n$ is the $n \times p$ matrix whose columns are given by $Y^n_1,\ldots,Y^n_p$. In particular, we may invert the relation 
\begin{align*}
\log \mathrm{det} (\Y^{\top}\Y) = 2 \left( \log \mathrm{Vol} (\Delta \Y )  + \log p! \right)
\end{align*}
to obtain various statements about the log-determinants of random matrices $\Y^{\top}$ whose columns are rotationally invariant random vectors. We remark that for our decomposition in \eqref{eq:decoupling} it is important that the rows of $\Y$ are independent rotationally invariant random vectors. If instead the columns of $\Y$ were rotationally invariant, one cannot separate the radius from the direction as in \eqref{eq:decoupling} even though $\Y^{\top}\Y$ and $\Y\Y^{\top}$ have the same non-zero eigenvalues. The phenomenon that the roles of rows and columns are not interchangeable is illustrated in Figure \ref{fig:2}. 

\begin{figure}[htb!]
  \centering
    \includegraphics[trim = 0in 0in 0in 0in, clip, scale=0.35]{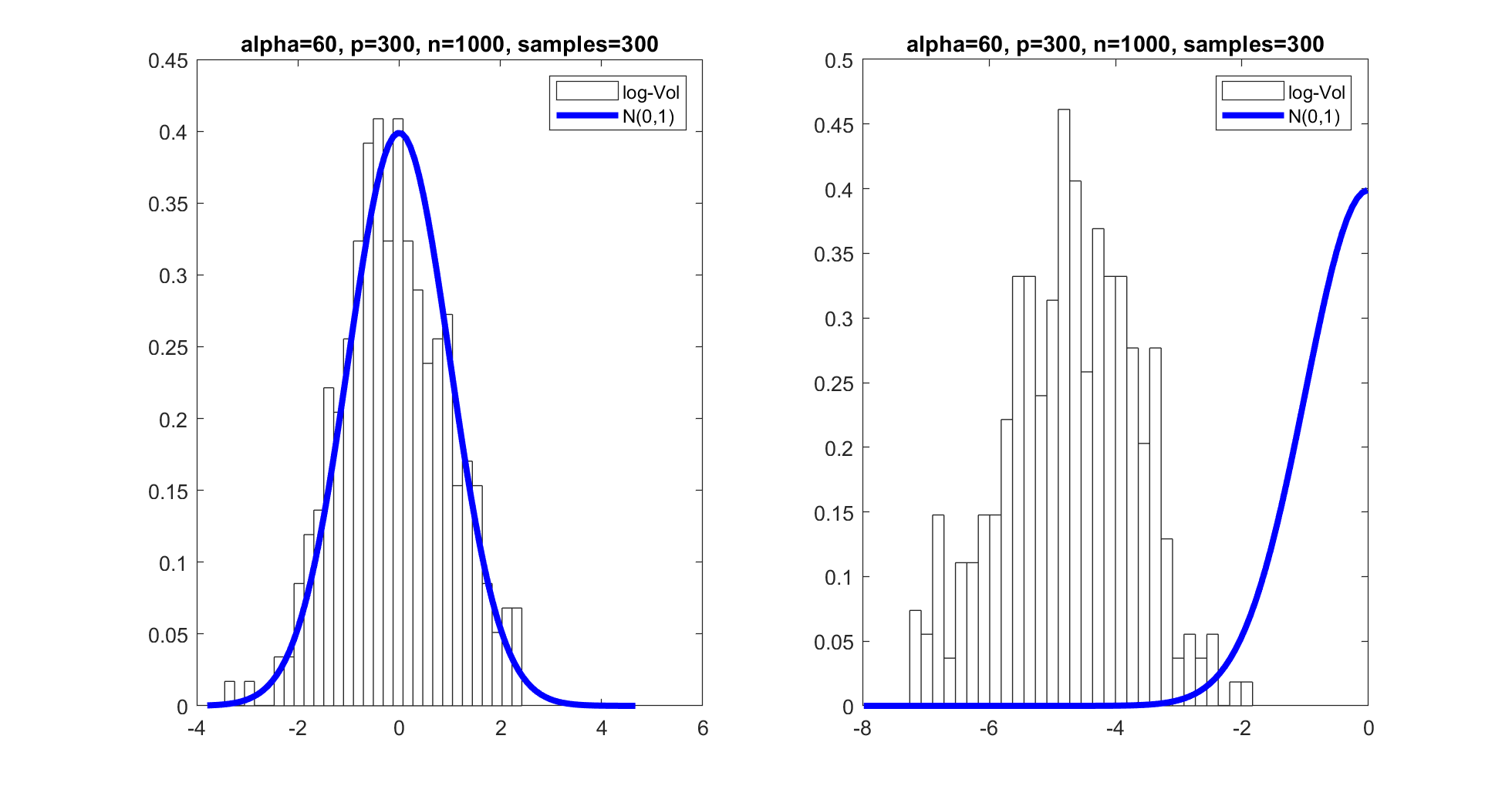}
\caption{Histograms of $300$ simulated $\log \mathrm{Vol}_p \left( \Delta \Y \right)$, properly centered and standardized (as in \cite{heiny:parolya:kurowicka:2021}), for $p=300$ and $n=1000$. Left: multivariate $t-$distributed rows of $\Y$ with 60 degrees of freedom. Right: multivariate $t-$distributed columns of $\Y$ with 60 degrees of freedom.}
\label{fig:2}
\end{figure}

Before discussing the applications of our results to the determinants of random matrices, we take a moment to highlight just a single result from the large body of work on the asymptotic distribution of the logarithms of such determinants \cite{goodman1963distribution,tao2012,bao2015, wang2018, heiny:parolya:kurowicka:2021}. Namely, Nyugen and Vu \cite{nguyen2014random} consider the log-determinant of an $n \times n$ random matrix $A^n$ with independent and identically distributed entries with zero mean, unit variance and finite fourth moment. They show that, as $\nto$, 
 \begin{align*}
\mathrm{d}_{KS} \left( \frac{ \log |\det (A^n )| - \frac{1}{2} \log (n-1)! }{ \sqrt{ \frac{1}{2} \log n } } , N \right)   = (\log n)^{ - 1/3 + o(1) }\,.
\end{align*}
Nguyen and Vu speculate that $(\log n)^{ - 1/3}$ could be the optimal rate of convergence for such a theorem, though suggest that this could be potentially improved to $(\log n)^{ - 1/2}$ with a finer correction for the expectation of the log determinant. It transpires that when the entries are further assumed to be independent standard Gaussians, the rate of convergence can be improved to $(\log n)^{ - 3/2}$. To this end, we require estimates on the mean and variance of $\log |\det(A^n)|$ that are fine up to constant order. To this end, let $\gamma := \lim_{n \to \infty} \left( \sum_{ k =1 }^n 1/k - \log n \right)$ denote the Euler-Mascheroni constant, and define the constants
\begin{align} \label{eq:c0 def}
c_0 := - \frac{\gamma}{2} - \frac{1}{2} \int_0^\infty \left( \frac{1}{2} - \frac{1}{\zeta} + \frac{1}{\e^\zeta - 1} \right) \frac{1}{\e^{\zeta/2} -1} \mathrm{d} \zeta
\end{align}
and
\begin{align} \label{eq:c1 def}
c_1 := \frac{\gamma}{2} + \frac{1}{4} \int_0^\infty \left( \frac{1}{2} - \frac{1}{\zeta} + \frac{1}{\e^\zeta - 1} \right) \frac{\zeta}{\e^{\zeta/2} -1} \mathrm{d} \zeta.
\end{align}

We believe the following result to be new.

\begin{thmalpha} \label{thm:gaussian matrix}
Let $n\in\N$ and $A^n$ be an $n \times n$ matrix whose entries are independent standard Gaussian random variables. Then, we have 
\begin{align*}
\mathrm{d}_{KS} \left( \frac{ \log |\det (A^n )| - ( \frac{1}{2} \log (n-1)! + c_0 ) }{ \sqrt{ \frac{1}{2} \log n + c_1 } } , N \right)  \leq \frac{ C }{ \log^{3/2} n },
\end{align*}
where $C\in(0,\infty)$ is an absolute constant.
\end{thmalpha}

Theorem \ref{thm:gaussian matrix} is proved directly in Section \ref{sec:gaussian proof}. A weaker version of Theorem \ref{thm:gaussian matrix}, without explicit estimates for the mean and variance of $\log |\det(A^n)|$ is actually an indirect consequence of a more general result concerning the log-determinants of random matrices whose columns are distributed according to a rotationally invariant probabillity density $\mu^n$ on $\mathbb{R}^n$. Namely, the following result is an immediate corollary of Theorem \ref{thm:main}, using \eqref{eq:det vol} to restate the result in terms of determinants of random matrices rather than volumes of random simplices. 

\begin{thmalpha} \label{thm:matrix main}
Let $A^{n,p}$ be an $n \times p$ matrix whose $p$ columns $Y_1^n,\ldots,Y_p^n$ are independent and identically distributed according to a probability density of the form $\mu^n$, with $\mu^n$ as in \eqref{eq:mu def} and $(G,H)$ a radially admissible pair. Then there is a constant $C_{G,H}\in(0,\infty)$ such that 
\begin{align*}
\mathrm{d}_{KS} \left( \frac{ \log \det \left( (A^{n,p})^{\top} A^{n,p}  \right)) - \mathbb{E} \left[  \log \det \left((A^{n,p})^{\top} A^{n,p}  \right)  \right]  }{ \sqrt{ \mathrm{Var}[ \log \det \left((A^{n,p})^{\top} A^{n,p}  \right)  ]  } } , N \right)  \leq C_{G,H} \frac{  \theta^2  }{ n(1-\theta) \left[ \log \frac{1}{1-\theta} - \theta \right]^{3/2}  }, 
\end{align*}
where $\theta := \theta(p,n):=\frac{p-1}{n}$.
\end{thmalpha}

That completes the section on random matrices.


\section{Extremely Gaussian behavior for spherical random simplices} \label{sec:spherical}

The chief focus of this section will be in analyzing the Gaussian behavior of the log-volume of random simplices whose vertices are uniformly distributed on the unit sphere. We begin in the next section by discussing the polar integration formula and radial laws.

\subsection{The polar integration formula and radial laws} \label{sec:polar}
Throughout we will use the following polar integration formula. Let $f:\mathbb{R}^n \to [0,\infty)$ be an integrable function on $\mathbb{R}^n$ depending only on the Euclidean norm, in the sense that $f(s) = \widetilde{f} (\twonorm{s})$ for some $\widetilde{f}:[0,\infty) \to [0,\infty)$. Then the polar integration formula states that
\begin{align} \label{eq:polar}
\int_{\mathbb{R}^n} f(s) \,\dint s = \frac{ 2 \pi^{n/2} }{ \Gamma(n/2 ) } \int_0^\infty r^{n-1} \widetilde{f}(r) \,\dint r. 
\end{align}
Given a Borel subset $A$ of $[0,\infty)$, define the Borel subset $\mathrm{rad}_n(A)$ of $\mathbb{R}^n$ by setting
\begin{align*}
\mathrm{rad}_n(A) := \{ s \in \mathbb{R}^n \,:\, \twonorm{s} \in A \}.
\end{align*}
Given any probability distribution $\mu$ on $\mathbb{R}^n$, we define the \emph{radial law} $\nu$ associated with $\mu$ to be the probability measure on $[0,\infty)$ defined by setting
\begin{align*}
\nu \left( A \right) := \mu \left( \mathrm{rad}_n(A) \right).
\end{align*}
We now record the following simple lemma on the radial laws of rotationally invariant distributions of standard form.

\begin{lemma} \label{lem:rad dist}
Let $\mu_n$ be a rotationally invariant probability distribution on $\mathbb{R}^n$ of the form
\begin{align*}
\mu_n (\dint s) = C_n g(\twonorm{s}) \e^{ -n h(\twonorm{s}) } \dint s.
\end{align*}
where $g,h:[0,\infty) \to [0,\infty)$ are measurable functions. 
Then the radial law $\nu_n$ associated with $\mu_n$ is given by
\begin{align*}
\nu_n(\dint r) = \frac{ 2 \pi^{n/2} C_n }{ \Gamma(n/2)} r^{n-1} g(r) \e^{ - n h(r) } \dint r.
\end{align*}
\end{lemma}

\begin{proof}
Let $A$ be a Borel subset of $[0,\infty)$. Then
\begin{align*}
\nu_n(A) = \mu_n \left( \mathrm{rad}_n(A) \right) = \int_{ \mathbb{R}^n } f(s) \,\dint s,
\end{align*}
where for $s\in\R^n$, $f(s) := \ind_{\mathrm{rad}_n(A)}(s) C_n g(\twonorm{s}) \e^{ - n h(\twonorm{s}) }$. The result follows after applying \eqref{eq:polar}.
\end{proof}

\subsection{Miles' identity}\label{subsec:miles identity}
Integral to our analysis is the distributional identity \eqref{eq:product rep} which is a consequence of the following proposition, which was recently given in Grote, Kabluchko and Th\"ale \cite[Theorem 2.4(d)]{GKT2019}, though similar identities date back (at least) to Miles \cite{M1971}.
\begin{proposition} \label{prop:spherical1}
Let $\Theta^n_1,\ldots,\Theta^n_p$ be points chosen independently and uniformly from the Euclidean unit sphere $\mathbb{S}^{n-1}$ in $\mathbb{R}^n$. Then we have the following identity in law
\begin{align} \label{dist id}
\Vol_p \left( \Delta( \Theta^n_1,\ldots,\Theta^n_p) \right) \stackrel{\dint}{=} \frac{1}{p!} \left( \prod_{ j = 1}^{p-1} \beta_{ \frac{n-j}{2} , \frac{j}{2} }\right)^{1/2}\,,
\end{align}
where $\left\{ \beta_{ \frac{n-j}{2} , \frac{j}{2} } : j = 1, \ldots, p-1 \right\}$ is a collection of independent random variables such that $\beta_{ \frac{n-j}{2} , \frac{j}{2} }$ is beta distributed with parameters $\left( \frac{n-j}{2} , \frac{j}{2} \right)$. 
\end{proposition}

It is immediate from Proposition \ref{prop:spherical1} that the log-volume of the spherical random simplex may be written as 
\begin{align} \label{eq:log spherical}
W_{n,p}^{ \mathrm{Sph}} := \log \mathrm{Vol}_p \Delta \left( \Theta_1^n,\ldots,\Theta^n_p \right) \stackrel{\dint}{=} - \log p! + \frac{1}{2} \sum_{ j = 1 }^{p-1} \log \beta_{(n-j)/2,j/2}.
\end{align}

\subsection{Polygamma functions} \label{sec:polygamma}
For complex $\zeta$ with positive real part, let $\Gamma(\zeta) := \int_0^\infty  u^{\zeta-1} \e^{ - u }\,\dint u$ be the gamma function. Then the $k$-th polygamma function is given by
\begin{align} \label{eq:polygamma}
\psi_k( \zeta) := \frac{ \d^{k+1}}{ \d\zeta^{k+1}} \log \Gamma( \zeta)\,, \qquad \zeta \in \C,\, \mathrm{Re}(\zeta) > 0.
\end{align}
The zero$^{\text{th}}$ polygamma function $\psi_0$, better known as the digamma function, has the following integral representation 
\begin{align} \label{eq:zero rep}
\psi_0(\zeta) := \int_0^\infty \left(\frac{\e^{-t}}{ t} - \frac{ \e^{ -  \zeta t} }{ 1 - \e^{ -t} }\right) \,\dint t, \qquad \zeta \in \C,\, \mathrm{Re}(\zeta) > 0, 
\end{align}
due to Gauss (see e.g. \cite[Section 1.4]{silverman1972special}). 
By differentiating through Gauss' integral representation \eqref{eq:zero rep} for $k \geq 1$ we have 
\begin{align} \label{eq:k rep}
\psi_k(\zeta) = (-1)^{k-1} \int_0^\infty \frac{ t^k e^{ - \zeta t}}{ 1 - e^{ - t} } \mathrm{d} t.
\end{align}
A simple calculation involving the gamma integral tells us that we have the sandwich inequality 
\begin{align} \label{eq:polygamma bound}
\frac{(k-1)!}{ \zeta^k} + \frac{k!}{2\zeta^{k+1}} \leq  \psi_k(\zeta) \leq \frac{(k-1)!}{\zeta^k} + \frac{k!}{\zeta^{k+1}} \qquad k \geq 1, \zeta \in [1/2,\infty)
\end{align} 
Finally, we note from \eqref{eq:k rep} that for $\zeta$ in $\mathbb{C}$ with $\mathrm{Re}(\zeta) > 0$ we have $|\psi_k(\zeta)| \leq |\psi_k(\mathrm{Re}(\zeta))|$. In particular, with $C_k = (k-1)! + 2k!$ we may extract from \eqref{eq:polygamma bound} the upper bound
\begin{align} \label{eq:polygamma bound 2}
|\psi_k(\zeta)| \leq \frac{C_k}{\mathrm{Re}(\zeta)^k} \qquad \zeta \in \mathbb{C}, \mathrm{Re}(\zeta) \geq 1/2.
\end{align} 

\subsection{Moments of log-beta random variables}
In this subsection, we provide all moments of the log-beta random variables in terms of combinatorial expressions involving the polygamma functions.
To this end, we need the one-dimensional Fa\`a di Bruno formula (see, e.g., \cite{JP2019FDB} for this and its multivariate form).  To set this up, recall that a \emph{partition} of $\{1,\ldots,k\}$ is a collection of disjoint subsets (called \emph{blocks}) of $\{1,\ldots,k\}$ whose union is equal to $\{1,\ldots,k\}$. Let $\mathcal{P}_k$ be the collection of set partitions of $\{1,\ldots,k\}$. For partitions $\pi$ in $\mathcal{P}_k$, let $\# \pi$ denote the number of blocks in $\pi$. For a block $\Gamma$ of some $\pi$, let $\# \Gamma$ denote the number of elements of $\{1,\ldots,k\}$ contained in $\Gamma$. 

Namely, if $k\in\N$ and $f,g:\mathbb{R} \to \mathbb{R}$ are $k$ times differentiable functions, then Fa\`a di Bruno's formula states that the $k$-th derivative of the composition $f\circ g$ is given by 
\begin{align*}
\frac{ \d^{k}}{ \d\zeta^{k}} f \left( g( \zeta) \right) = \sum_{ \pi \in \mathcal{P}_k } f^{ (\# j ) } \left( g(\zeta) \right) \prod_{ \Gamma \in \pi} g^{ (\# \Gamma)} (\zeta),
\end{align*}
where for $j\in\N$, $f^{(j)}$ and $g^{(j)}$ denote the $j$-th derivatives of $f$ and $g$ respectively. We note that in particular, when $f(\zeta) = \e^{\zeta}$, we have
\begin{align} \label{eq:exp fdb}
\frac{ \d^{k}}{ \d\zeta^{k}} \e^{ g(\zeta) } = \e^{ g(\zeta) }  \sum_{ \pi \in \mathcal{P}_k } \prod_{ \Gamma \in \pi} g^{ (\# \Gamma)} (\zeta).
\end{align}

We are now equipped to give a combinatorial representation for the moments of the logarithm of a beta random variable in terms of set partitions and the polygamma functions.

\begin{lemma} \label{lem:beta moments}
Let $\beta_{\zeta,\eta}$ be beta distributed with parameters $(\zeta,\eta)$. Then 
\begin{align*}
\E\Big[ \left( \log \beta_{\zeta,\eta} \right)^k \Big] = \sum_{ \pi \in \mathcal{P}_k} \prod_{ \Gamma \in \pi} q_{\# \Gamma}(\zeta,\eta),
\end{align*}
where for integers $j \in\N$, $q_j(\zeta,\eta):= \psi_{j -1 }( \zeta) - \psi_{ j - 1} (\zeta + \eta)   $.
\end{lemma}

\begin{proof}
First we make the observation that
\begin{align*}
\mathbb{E}[ \log \beta_{\zeta, \eta} ] &= \frac{ \Gamma(\zeta + \eta)}{ \Gamma(\zeta) \Gamma(\eta) } \int_0^1 \left( \log s \right)^k  ~ s^{ \zeta - 1} (1-s)^{ \eta - 1} \d s\\
&= \frac{ \Gamma(\zeta + \eta)}{ \Gamma(\zeta) \Gamma(\eta) } \int_0^1 \frac{ \d^k}{ \d\zeta^k} s^{ \zeta - 1} (1-s)^{ \eta - 1} \d s.
\end{align*}
In particular, by taking the derivative outside the integral, we may write
\begin{align*}
\mathbb{E}[ \log \beta_{\zeta, \eta} ]  = \e^{ - g(\zeta) } \frac{ \d^k}{ \d\zeta^k} \e^{ g(\zeta)},
\end{align*}
where $g(\zeta) := \log \Gamma(\zeta) - \log  \Gamma (\zeta+\eta) $. The result follows by using \eqref{eq:exp fdb} and the definition \eqref{eq:polygamma} of the polygamma functions. 
\end{proof}

Our second lemma gives us the centered moments.
\begin{lemma} \label{lem:beta centered moments}
Let $k\in\N$ and $\mathcal{Q}_k$ be the collection of set partitions of $\{1,\ldots,k\}$ containing no singletons. Assume that $\beta_{\zeta,\eta}$ is beta distributed with parameters $(\zeta,\eta)$. Then
\begin{align*}
\E\Big[ \left( \log \beta_{\zeta,\eta} - \mathbb{E} [  \log \beta_{\zeta,\eta}  ] \right)^k \Big] = \sum_{ \pi \in \mathcal{Q}_k} \prod_{ \Gamma \in \pi} q_{\# \Gamma}(\zeta,\eta),
\end{align*}
where for integers $j \in\N$, $q_j(\zeta,\eta) := \psi_{j -1 }( \zeta) - \psi_{ j - 1} (\zeta + \eta) $.
\end{lemma}
\begin{proof}
We begin with the observation that 
\begin{align*}
\E\Big[ \left( \log \beta_{\zeta,\eta} - \mathbb{E} [  \log \beta_{\zeta,\eta}  ] \right)^k \Big]  = \sum_{ \mathcal{S} \subseteq \{1,\ldots,k\} } (-1)^{ k - \# \mathcal S }  q_{1}^{ k - \# \mathcal S} \E[ (\log \beta_{ \zeta,\eta})^{\#\mathcal S} ],
\end{align*}
where we wrote $q_j:=q_j(\zeta,\eta)$ for simplicity.
For each subset $\mathcal{S}$ of $\{1,\ldots,k\}$, we may expand $ \E[ (\log \beta_{ \zeta,\eta})^{\#\mathcal S} ]$ using Lemma \ref{lem:beta moments} so that
\begin{align} \label{eq:comb rep}
\E\Big[ \left( \log \beta_{\zeta,\eta} - \mathbb{E} [  \log \beta_{\zeta,\eta}  ] \right)^k \Big]  = \sum_{ \mathcal{S} \subseteq \{1,\ldots,k\} }  (-1)^{ k - \# \mathcal S } q_1^{ k - \# \mathcal S} \sum_{ \pi \in \mathcal{P}_{ \mathcal{S}}} \prod_{  \Gamma \in \pi} q_{\# \Gamma}.
\end{align}
Each partition $\pi$ of $\mathcal{S}$ has a canonical extension $\bar{\pi}$ to $\{1,\ldots,k\}$ by letting 
\begin{align*}
\bar{\pi} = \pi \cup \left\{ \{x\} : x \in \{1,\ldots,k\} - \mathcal S \right\}.
\end{align*}
 Let $A(\pi)$ be the set of $x \in \{1,\ldots,k\}$ such that the singleton $\{x\}$ is a block of $\pi$. It follows that $\mathcal{T} = k - \mathcal{S}$ is a subset of $A(\bar{\pi})$. In particular, reindexing the sum in \eqref{eq:comb rep}, we have 
\begin{align} \label{eq:bar sum}
\E\Big[ \left( \log \beta_{\zeta,\eta} - \mathbb{E} [  \log \beta_{\zeta,\eta}  ] \right)^k \Big] =  \sum_{ \bar{\pi} \in \mathcal{P}_k } \left( \sum_{ \mathcal{T} \subseteq A( \bar{\pi} ) } (-1)^{ \# \mathcal{T} } \right)  \prod_{  \Gamma \in \bar{\pi}} q_{\# \Gamma}.
\end{align}
Now note that
\begin{align*}
\sum_{ \mathcal{T}  \subseteq  A } (-1)^{ \# \mathcal{T} } = 
\begin{cases}
1  \qquad &\text{if $A$ is empty},\\
0 \qquad & \text{otherwise}.
\end{cases}
\end{align*}
It follows that the sum in \eqref{eq:bar sum} is supported only on partitions $\bar{\pi}$ in $\mathcal{P}_k$ such that $A(\bar{\pi})$ is empty, i.e., contains no singletons.
\end{proof}

The following lemma collects together some information on the first three moments of $\log \beta_{\zeta,\eta}$, and is an immediate consequence of Lemmas \ref{lem:beta moments} and \ref{lem:beta centered moments}.

\begin{lemma} \label{lem:first three}
Assume that $\beta_{\zeta,\eta}$ is beta distributed with parameters $(\zeta,\eta)$. The mean and variance of $\log \beta_{\zeta,\eta}$ are given by
\begin{align*}
\mathbb{E} \left[ \log \beta_{\zeta,\eta} \right]  = \psi_0( \zeta) - \psi_{0} (\zeta + \eta) \quad \text{ and }\quad \Var \left[ \log \beta_{\zeta,\eta} \right]  = \psi_1( \zeta) - \psi_{1} (\zeta + \eta)  .
\end{align*}
Moreover, we have the following upper bound on the centered absolute third moment 
\begin{align*}
\mathbb{E} \left[ \left| \log \beta_{\zeta,\eta} - \mathbb{E} \left[ \log \beta_{\zeta,\eta} \right]  \right|^3 \right]   \leq \left( \psi_3( \zeta) - \psi_{3} (\zeta + \eta)  + \left( \psi_1( \zeta) - \psi_{1} (\zeta + \eta)  \right)^2\right)^{3/4}.
\end{align*}
\end{lemma}
\begin{proof}
The equations for the mean and variance follow from respectively setting $k = 1$ in Lemma \ref{lem:beta moments} and $k=2$ in Lemma \ref{lem:beta centered moments}. The upper bound for the centered absolute third moment is obtained by setting $k=4$ in Lemma \ref{lem:beta centered moments} and using Lyapunov's inequality. 
\end{proof}

If $W_{n,p}^{\mathrm{Sph}}$ is the log-volume of a spherical random simplex associated with $p$ points sampled independently and uniformly from $\mathbb{S}^{n-1}$, then by Lemma \ref{lem:first three} and \eqref{eq:log spherical} we have
\begin{align*}
\mu_{n,p}^{ \mathrm{Sph}} := \mathbb{E}\Big[ W_{n,p}^{\mathrm{Sph}} \Big] = - \log (p!) + \frac{1}{2} \sum_{ j = 1}^{p-1} \left(\psi_0 \left( \frac{n-j}{2} \right) - \psi_0 \left( \frac{n}{2} \right) \right)
\end{align*}
and 
\begin{align*}
\left( \sigma_{n,p}^{\mathrm{Sph}} \right)^2 := \mathrm{Var}\Big[ W_{n,p}^{\mathrm{Sph}} \Big] =  \frac{1}{4} \sum_{ j = 1}^{p-1} \left(\psi_1 \left( \frac{n-j}{2} \right) - \psi_1 \left( \frac{n}{2} \right) \right).
\end{align*}

At several stages below we will require the following lower bound on the variance $\left(\sigma_{n,p}^{\mathrm{Sph}}\right)^2$, which follows easily from \eqref{eq:polygamma bound}.

\begin{cor} \label{cor:var control}
Let $p,n\in\N$ with $p\leq n$. Then we have 
\begin{align} \label{eq:explicit sigma}
\left( \sigma_{n,p}^{\mathrm{Sph}} \right)^2 \geq \frac{1}{2} \left[ - \log \left( 1 -  \frac{p-1}{n} \right) - \frac{p-1}{n} \left( 1 + \frac{3}{2n } \right)  \right].
\end{align}
Setting $\theta := \theta(p,n):= \frac{p-1}{n}$, whenever $p \geq 7$ we have the rougher bound
\begin{align} \label{eq:smoother}
\left( \sigma_{n,p}^{\mathrm{Sph}} \right)^2 \geq \frac{1}{4} \left( \log \frac{1}{1- \theta} - \theta \right).
\end{align}
\end{cor}

\begin{proof} Using \eqref{eq:polygamma bound} to obtain the first inequality below we have 
\begin{align*}
\left( \sigma_{n,p}^{\mathrm{Sph}} \right)^2 &:= \frac{1}{4} \sum_{j=1}^{p-1} \left( \psi_1\left( \frac{n-j}{2} \right) - \psi_1\left( \frac{n}{2} \right) \right)\\
&\geq \frac{1}{4} \sum_{ j = 1}^{p-1} \left( \frac{2}{n-j} + \frac{1}{(n-j)^2} - \frac{2}{n} - \frac{4}{n^2} \right) \\
&\geq \frac{1}{4} \sum_{ j = 1}^{p-1} \left( \frac{2}{n-j}  - \frac{2}{n} - \frac{3}{n^2} \right) \\
&= -\frac{1}{4} (p-1) \left( \frac{2}{n} + \frac{3}{n^2} \right) + \frac{1}{2} \int_{n-p+1}^{n} \frac{ \mathrm{d} \zeta}{ \lfloor \zeta \rfloor},
\end{align*}
where $\lfloor \zeta \rfloor$ is the largest integer less than $\zeta$. Using the fact that for $\zeta>0$, $\frac{1}{\lfloor \zeta \rfloor} \geq \frac{1}{\zeta}$, and then performing the resulting integral, the bound \eqref{eq:explicit sigma} follows. 

As for the second bound, suppose $p \geq 7$. Now rewriting \eqref{eq:explicit sigma} to obtain the first inequality below, and using the fact that $p \geq 7$ to obtain the second, we have 
\begin{align*}
\left( \sigma_{n,p}^{\mathrm{Sph}} \right)^2 &\geq \frac{1}{2} \left[ \log \frac{1}{1-\theta} - \theta - \frac{3}{ 2(p-1)} \theta^2 \right] \geq \frac{1}{2} \left[ \frac{1}{2} \left( \log \frac{1}{1-\theta} - \theta \right) + \frac{1}{2} \left( \log \frac{1}{1-\theta} - \theta - \theta^2/2\right) \right]. 
\end{align*}
The result now follows from using the inequality $ \log \frac{1}{1-\theta} - \theta - \theta^2/2 \geq 0$. 
\end{proof}

That completes the section on the moments of the log-gamma random variables. 
In the next section we undertake a careful analysis of the characteristic function of the log-beta random variable, which is the most delicate step in proving Theorem \ref{thm:spherical berry}. 

\subsection{The characteristic function of the log-beta random variable} \label{sec:log beta char}

Our proof of Theorem \ref{thm:spherical berry} involves a Fourier-analytic approach based on a careful analysis of the characteristic function of $W_{n,p}^{\mathrm{Sph}}$. We begin with the following lemma giving a useful representation for the characteristic function of a recentering of $\log \beta_{(n-j)/2,j/2}$. 

\begin{lemma}  \label{lem:key moments}
For $j,n\in\N$ such that $j < n$ let $ \beta_{ \frac{n-j}{2} , \frac{j}{2} }$ be a beta distributed random variable with shape parameters $(\frac{n-j}{2} , \frac{j}{2} )$, let $Y_{n,j} := \log \beta_{ \frac{n-j}{2} , \frac{j}{2} }$, and set $V_{n,j} := \frac{ Y_{n,j} - \mathbb{E}[ Y_{n,j}] }{ \sqrt{\mathrm{Var}[Y_{n,j}]} }$. Then, for all $t\in\R$, the characteristic function of $V_{n,j}$ is given by 
\begin{align*}
\varphi_{n,j}(t) &:= \mathbb{E} \left[ \e^{ i t V_{n,j}} \right]\\
&= \exp \left\{ - \frac{t^2}{2} + \frac{i}{2} \int_{0}^{t/\sigma_{n,j}} \int_0^{t_1}  \int_0^{t_2} \int_0^{j} \psi_3 \left( \frac{n-s}{2} + i t_3 \right)  \mathrm{d}s \,  \mathrm{d}t_3 \, \mathrm{d}t_2 \, \mathrm{d}t_1 \right\},
\end{align*}
where $\psi_3$ is as in \eqref{eq:polygamma} and $\sigma_{n,j}^2 := \psi_1 \left( \frac{n-j}{2} \right)  - \psi_1 \left( \frac{n}{2} \right)$. 
\end{lemma}

\begin{proof}
We begin by studying the characteristic function of $Y_{n,j}$ for $t\in \R$. It is a straightforward computation using the definition of the beta-integral to see that  
\begin{align}\label{eq:characteristic function log-beta}
\mathbb{E} [ \e^{ i t Y_{n,j}} ]  = \frac{ \Gamma(n/2) }{ \Gamma((n-j)/2) \Gamma(j/2) } \int_0^1 s^{ it + (n-j)/2 - 1} (1 - s)^{ j /2 - 1 } \mathrm{d} s  = \frac{ \Gamma \left( \frac{n-j}{2} + i t \right)}{ \Gamma \left( \frac{n}{2} + i t \right) } \frac{ \Gamma\left( \frac{n}{2} \right) }{ \Gamma \left(\frac{n-j}{2}\right) }.
\end{align}
By integrating in the complex plane and using the definition of the digamma function $\psi_0$, we have 
\begin{align*}
\log \Gamma \Big( \frac{n-j}{2} + it \Big) - \log \Gamma\Big( \frac{n}{2} + it\Big) = - \frac{1}{2} \int_0^j \psi_0 \left( \frac{n-s}{2} + it  \right) \mathrm{d} s
\end{align*} 
and similarly we obtain (by simply setting $t=0$ in the previous display and changing the sign) that
\[
\log \Gamma\Big( \frac{n}{2}\Big) - \log \Gamma \Big( \frac{n-j}{2} \Big) =  \frac{1}{2} \int_0^j \psi_0 \left( \frac{n-s}{2}  \right) \mathrm{d} s.
\] 
In view of \eqref{eq:characteristic function log-beta}, this means that we may write
\begin{align*}
\mathbb{E} [ \e^{ i t Y_{n,j}} ]  = \exp \left\{  - \frac{1}{2} \int_0^j \psi_0 \left( \frac{n-s}{2} + i t \right) - \psi_0 \left( \frac{n-s}{2}  \right)\mathrm{d}s  \right\}.
\end{align*} 
Performing a second integral in the complex plane, as using the definition of $\psi_1$, we obtain
\begin{align*}
\mathbb{E} [ \e^{ i t Y_{n,j}} ]  = \exp \left\{  - \frac{i}{2} \int_0^t \int_0^j ~ \psi_1 \left( \frac{n-s}{2} + i t_1 \right)   \mathrm{d}s \, \mathrm{d}t_1   \right\}.
\end{align*} 
We now turn to extracting the characteristic function of $V_{n,j}$ from that of $Y_{n,j}$. First of all from the definition of $V_{n,j}$ we plainly have 
\begin{align} \label{eq:old package}
\mathbb{E} \left[ \e^{ i t V_{n,j} } \right] = \exp \left\{ - i \frac{t}{ \sigma_{n,j} }  \mu_{n,j} - \frac{i}{2} \int_0^{t/\sigma_{n,j}}  \int_0^j ~ \psi_1 \left( \frac{n-s}{2} + i t_1 \right)    \mathrm{d}s  \,  \mathrm{d}t_1   \right\} ,
\end{align}
where $\mu_{n,j}$ and $\sigma_{n,j}^2$ denote respectively the mean and variance of $\log \beta_{(n-j)/2,j/2}$, which were identified in Lemma \ref{lem:first three} above. We now note that we may usefully represent $\mu_{n,j}$ as an integral via
\begin{align} \label{eq:new package}
\mu_{n,j} = \psi_0 \left( \frac{n-j}{2} \right) - \psi_0 \left( \frac{n}{2} \right) = - \frac{1}{2} \int_0^j  \psi_1 \left( \frac{n-s}{2} \right)\mathrm{d} s,
\end{align} 
so that plugging \eqref{eq:new package} into \eqref{eq:old package} to obtain the first equality below, and performing another integration to obtain the second, we have 
\begin{align} \label{eq:super package}
\mathbb{E} \left[ \e^{ i t V_{n,j} } \right] &= \exp \left\{  - \frac{i}{2} \int_0^{t/\sigma_{n,j}} \int_0^j  \psi_1 \left( \frac{n-s}{2} + i t_1 \right)  -  \psi_1 \left( \frac{n-s}{2} \right)  \mathrm{d}s \, \mathrm{d}t_1   \right\} \nonumber \\
&= \exp \left\{ \frac{1}{2} \int_0^{t/\sigma_{n,j}}   \int_0^{t_1} \mathrm{d} t_2 \int_0^j  \psi_2 \left(   \frac{n-s}{2} + i t_2 \right)  \mathrm{d}s \, \mathrm{d}t_1 \right\} .
\end{align}
Using $\sigma_{n,j}^2 = \psi_1 \left( \frac{n-j}{2} \right ) - \psi_1 \left( \frac{n}{2 } \right) = - \frac{1}{2} \int_0^j   \psi_2 \left( \frac{n-s}{2} \right) \mathrm{d}s$ it can be checked that
\begin{align} 
\int_0^{t/\sigma_{n,j}} \int_0^{t_1} \int_0^j  \psi_2 \left(   \frac{n-s}{2} \right) \mathrm{d}s \, \mathrm{d} t_2 \, \mathrm{d}t_1  &= - t^2.
\end{align}
Plugging this into \eqref{eq:super package}, we have
\begin{align} 
\mathbb{E} \left[ \e^{ i t V_{n,j} } \right] &= \exp \left\{  - t^2/2 +  \frac{1}{2} \int_0^{t/\sigma_{n,j}} \int_0^{t_1} \int_0^j  \psi_2 \left(   \frac{n-s}{2} + i t_2 \right)  - \psi_2 \left(   \frac{n-s}{2} \right) \mathrm{d}s \, \mathrm{d} t_2 \, \mathrm{d}t_1  \right\}.
\end{align}
The result follows from a final integration step.
\end{proof}

We now turn to studying the characteristic function of the sum of the log-beta random variables.
First we note that if $\widetilde{W}_{n,p} := \big( W_{n,p}^{\mathrm{Sph}} - \mu_{n,p}^{\mathrm{Sph}}\big)/\sigma_{n,p}^{\mathrm{Sph}}$, then with $\varphi_{n,j}$ as in Lemma \ref{lem:key moments} the characteristic function of $\widetilde{W}_{n,p}$ is given by 
\begin{align} \label{eq:orange}
\phi_{n,p}(t) &:= \mathbb{E} \left[ \e^{  i t \widetilde{W}_{n,p} } \right] = \prod_{ j =1}^{p-1} \varphi_{n,j} \left( \frac{ t \sigma_{n,j} }{ 2 \sigma_{n,p}^{\mathrm{Sph}} } \right) \nonumber\\
&= \exp \left( - t^2/2 + \frac{i}{2} \int_{ 0 < t_3 < t_2 < t_1 < t/(2\sigma_{n,p}^{\mathrm{Sph})} } \sum_{ j = 1}^{p-1} \int_0^j \psi_3 \left( \frac{n-s}{2} +  i t_3 \right)\, \mathrm{d}s \,\mathrm{d}t_1\, \mathrm{d}t_2\, \mathrm{d}t_3 \right),
\end{align}
where the final line above follows from a brief calculation using Lemma \ref{lem:key moments}.

Clearly, by virtue of the centering, the random variable $\widetilde{W}_{n,p}$ has zero mean and unit variance. 
The following lemma compares the logarithms of the characteristic funtions of $\widetilde{W}_{n,p}$ and a standard Gaussian random variable, where we recall that the latter is $t\mapsto \e^{- t^2/2}$.

\begin{lemma} \label{lem:char bound}
Let $p,n\in\N$ with $p\leq n$ and $\phi_{n,p}$ be the characteristic function of $\widetilde{W}_{n,p}$. Then, for all $t\in\R$,
\begin{align*}
\left| \log \phi_{n,p}(t) + t^2/2 \right| \leq \varepsilon_{n,p} |t|^3,
\end{align*}
where
\begin{align} \label{eq:varepsilon}
\varepsilon_{n,p} :=   \frac{7}{96} \frac{ 1 }{ (\sigma_{n,p}^{\mathrm{Sph}})^3}  \sum_{ j =1}^{p-1} \left[ \frac{1}{(n-j)^2} - \frac{1}{n^2} \right].
\end{align} 
\end{lemma}

\begin{proof}
Using \eqref{eq:orange} to obtain the first inequality below, and using the upper bound in \eqref{eq:polygamma bound 2} with $k=3$ (and hence $C_3 = 14$) to obtain the second, we have 
\begin{align*}
\left| \log \phi_{n,p}(t) + t^2/2 \right| & \leq \frac{1}{2} \int_{ 0 < t_3 < t_2 < t_1 < t/(2\sigma_{n,p}^{\mathrm{Sph})} } \sum_{ j = 1}^{p-1} \int_0^j \left| \psi_3 \left( \frac{n-s}{2} +  i t_3 \right) \right| \,\mathrm{d}s\, \mathrm{d}t_1\, \mathrm{d}t_2\, \mathrm{d}t_3 \\
&\leq \frac{1}{2} \int_{ 0 < t_3 < t_2 < t_1 < t/(2\sigma_{n,p}^{\mathrm{Sph}}) } \sum_{ j = 1}^{p-1} \int_0^j \frac{ 14 }{ (n-s)^3}\,  \mathrm{d}s\, \mathrm{d}t_1\, \mathrm{d}t_2\, \mathrm{d}t_3 .
\end{align*}
The latter integrand above is independent of $t_1,t_2,t_3$. In particular, since the simplex $\{ (t_1,t_2,t_3) \in \mathbb{R}^3 : 0 < t_3 < t_2 < t_1 < a \} $ has volume $a^3/6$, we have 
\begin{align*}
\left| \log \phi_{n,p}(t) + t^2/2 \right| & \leq \frac{7}{48} \frac{ |t|^3}{ (\sigma_{n,p}^{\mathrm{Sph}})^3}   \sum_{ j = 1}^{p-1} \int_0^j \frac{ 1 }{ (n-s)^3}  \mathrm{d}s .
\end{align*}
The result in question follows by performing the $s$ integral. 
\end{proof}

Our next lemma appraises the factor $\varepsilon_{n,p}$ featuring in Lemma \ref{lem:char bound}.

\begin{lemma}
Let $\varepsilon_{n,p}$ be as in Lemma \ref{lem:char bound}. Then, whenever $p \geq 7$, we have 
\begin{align}
\varepsilon_{n,p} \leq  \frac{7}{4} \frac{  \theta^2  }{ n(1-\theta) \left[ \log \frac{1}{1-\theta} - \theta \right]^{3/2}  },
\end{align}
where $\theta:=\theta(p,n) := \frac{p-1}{n}$. 
\end{lemma}

\begin{proof}
We would like to bound the sum occuring in \eqref{eq:varepsilon}. To this end, we note that
\begin{align} \label{eq:coal2}
\sum_{ j =1}^{p-1} \left[ \frac{1}{(n-j)^2} - \frac{1}{n^2} \right] &= - \frac{p-1}{n^2} + \frac{1}{(n-p+1)^2} + \sum_{ j =1}^{p-2} \frac{1}{(n-j)^2} \nonumber \\
&\leq  - \frac{p-1}{n^2} + \frac{1}{(n-p+1)^2} + \int_{n-p+1}^{n-1} \frac{ \mathrm{d}s }{ s^2} \nonumber \\
&=  - \frac{p-1}{n^2} + \frac{1}{(n-p+1)^2} + \frac{1}{n-p+1} - \frac{1}{n-1} \nonumber\\
&= \frac{1}{n} \left[ \frac{1}{ 1- \theta} - 1 - \theta \right] + \frac{1}{n^2(1-\theta)^2} - \left( \frac{1}{n-1} - \frac{1}{n} \right) \nonumber \\
&\leq \frac{1}{n} \left[ \frac{1}{ 1- \theta} - 1 - \theta \right] + \frac{1}{n^2} \left( \frac{1}{(1 - \theta)^2} - 1 \right)  \nonumber \\
&\leq \frac{1}{n} \left[ \frac{1}{ 1- \theta} - 1 - \theta \right] + \frac{1}{n^2} \frac{3 \theta}{(1 - \theta)^2} \nonumber \\
&= \frac{\theta }{ n(1-\theta) } \left[ \theta +  \frac{3}{ n(1-\theta)} \right]. 
\end{align}
Whenever $7 \leq p \leq n$, we have $\frac{1}{n} \leq \theta := \frac{p-1}{n} \leq 1 - \frac{1}{n}$. In particular, for all such $\theta$ we have $n \theta ( 1- \theta) \geq 1 - 1/n \geq 6/7$ so that $\frac{3}{n(1-\theta)} \leq \frac{7}{2} \theta$. Thus, from \eqref{eq:coal2} we obtain  
\begin{align} \label{eq:coal3}
\sum_{ j =1}^{p-1} \left[ \frac{1}{(n-j)^2} - \frac{1}{n^2} \right] \leq \frac{ 9 \theta^2}{ 2 n(1-\theta)}.
\end{align}
In particular, combining \eqref{eq:coal3} with \eqref{eq:varepsilon} we have 
\begin{align*}
\varepsilon_{n,p} \leq \frac{7}{96} \frac{1}{ (\sigma_{n,p}^{\mathrm{Sph}})^3 } \frac{9}{2}  \frac{ \theta^2}{n(1-\theta)} = \frac{7}{16}  \frac{1}{ (\sigma_{n,p}^{\mathrm{Sph}})^3 } \frac{ \theta^2}{n(1-\theta)}.
\end{align*}
The result follows by combining the bound \eqref{eq:coal3} with \eqref{eq:smoother} in the definition \eqref{eq:varepsilon}. 
\end{proof}

The following lemma is the final step in the proof. This technique is well known, appearing in various proofs of the Berry-Esseen theorem (see, e.g., Petrov \cite[Chapter 5]{petrov:1975}).

\begin{lemma} \label{lem:berry technique}
Let $f:\mathbb{R} \to \mathbb{C}$ be a function that satisfies 
\begin{align}\label{eq:bound log f}
\left| \log f(t) + t^2/2 \right| \leq \varepsilon |t|^3
\end{align}
for all $t\in\R$. Then, for all $|t| \leq 1/(4\varepsilon)$ we have
\begin{align*}
\left| f(t) - \e^{ - t^2/2} \right| \leq \varepsilon |t|^3 \e^{ - t^2/4}.
\end{align*}
\end{lemma}

\begin{proof}
For $z\in\C$ we have the inequality $|\e^z-1| \leq |z| \e^{|z|}$. In particular, using the assumption \eqref{eq:bound log f},
\begin{align*}
\left| f(t) - \e^{ - t^2/2} \right| = \e^{ - t^2/2} \left| f(t) \e^{t^2/2} - 1  \right| \leq |\log f(t) + t^2/2| \e^{ |\log f(t) + t^2/2| } \leq \varepsilon |t|^3 \exp \left( - \frac{t^2}{2} + \varepsilon |t|^3 \right).
\end{align*}
The result follows by noting that whenever $|t| \leq 1/4\varepsilon$, $t^2/2 - \varepsilon|t|^3 \geq t^2/4$.
\end{proof}

\subsection{Proof of Theorem \ref{thm:spherical berry}} \label{sec:spherical proof}

We are now ready to prove Theorem \ref{thm:spherical berry}.

\begin{proof}[Proof of Theorem \ref{thm:spherical berry}]
Theorem \ref{thm:spherical berry} follows from the statement 
\[ 
 \mathrm{d}_{KS}\left( \widetilde{W}_{N,p} , N \right)\leq 16\, \varepsilon_{n,p},
\]
which we now prove. By the Berry smoothing inequality (see, e.g., \cite[Section 7.4]{chung2001course}), the Kolmogorov-Smirnov distance between $\widetilde{W}_{n,p}$ and a standard Gaussian random variable $N$ may be bounded via
\begin{align} \label{eq:berry}
\mathrm{d}_{KS}\left( \widetilde{W}_{N,p} , N \right)\leq  \frac{1}{\pi}  \int_{-T}^T \frac{ |\phi_{n,p}(t) - \e^{ -t^2/2} | }{ |t|} \,\mathrm{d}t + \frac{10}{\pi T},
\end{align}
for any $T > 0$. Setting $T := (4 \varepsilon_{n,p})^{-1}$ and appealing to Lemmas \ref{lem:berry technique} and \ref{lem:char bound}, we have 
\begin{align} 
\mathrm{d}_{KS}\left( \widetilde{W}_{n,p} , N \right) \leq  \frac{ \varepsilon_{n,p} }{\pi}  \int_{-T}^T t^2 \e^{ - t^2/4} \mathrm{d}t + \frac{40}{\pi } \varepsilon_{n,p} ,
\end{align}
The result follows by using the fact that $\int_{-\infty}^\infty t^2 \e^{ - t^2/2} \mathrm{d}t = 4 \sqrt{\pi} \leq 8$, and then using that $48/\pi \leq 16$.
\end{proof}

\section{Central limit theory and the Laplace method} \label{sec:laplace}

\subsection{Statement}

With a view to proving Theorem \ref{thm:laplace berry}, in this section we will be considering probability density functions of the form
\begin{align} \label{eq:rho def 1}
\rho^n(\d x) = \frac{1}{Z^n} g(x) \exp \left( - n h(x) \right) \mathrm{d} x,\quad n\in\N,
\end{align}
where the $Z^n\in(0,\infty)$, $n\in\N$ are normalization constants, and the ordered pair of functions $g,h : [0,\infty) \to [0,\infty)$ is \emph{admissible} in the sense of Definition \ref{df:admissible}. Recall in particular that the function $h$ has a global minimum at a point $x_0 \in \mathbb{R}$.

By changing the normalization constant $Z^n$ if necessary, we may assume without loss of generality that $h(x_0) = 0$ and $g(x_0) = 1$. Moreover, since the random variables in the statement of Theorem \ref{thm:laplace berry} are recentered, whenever the statement of Theorem \ref{thm:laplace berry} holds for a density $\frac{1}{Z^n} g(x) e^{ - n h(x) }$ it also holds for the rescaled and recentered density $\frac{\lambda}{Z^n} g( \lambda x + \mu ) e^{ - n h(\lambda x + \mu ) }$. In particular, we may assume without loss of generality that the global minimum occurs at zero, i.e. $x_0 = 0$, and that $h''(x_0) =1 $.

In summary, without loss of generality we restrict ourselves to considering densities of the form
\begin{align} \label{eq:rho original}
\rho^n(x) := \frac{Q^n}{ \sqrt{2 \pi}} \left( 1 + q(x) \right) \exp \left( - n \left( x^2/2 + r(x) \right) \right),\qquad x \in \mathbb{R},
\end{align}
where $Q^n\in(0,\infty)$ is a normalizing constant and where by the assumptions of Definition \ref{df:admissible}, $r,q:\mathbb{R} \to \mathbb{R}$ have the following properties: first, by part (b) of Definition \ref{df:admissible} there exists some $\delta > 0$ such that 
\begin{align} \label{eq:local bounds}
|r(x)| & \leq \frac{1}{4 \delta} |x|^3 \qquad \text{and} \qquad |q(x) | \leq \frac{1}{4 \delta} |x|  \qquad \text{for } x \in [-\delta,\delta],
\end{align}
where as by part (c) there exist constants $\alpha,c,C \in (0,\infty)$ such that  
\begin{align} \label{eq:tail bounds}
h(x) \geq c  \log (1 + |x|) \qquad \text{and} \qquad g(x) \leq C ( 1 + |x|^\alpha ) \qquad \text{for } x \in \mathbb{R}-[-\delta,\delta].
\end{align}
Again, without loss of generality, (since the random variables in the statement of Theorem \ref{thm:laplace berry} are centered), we may change variable $x \mapsto x/\sqrt{n}$, so that we consider for $n\in\N$ densities $J^n:\R\to[0,\infty)$ of the form 
\begin{align}
J^n(x) = D^n \frac{ \e^{ - x^2/2} }{ \sqrt{2\pi}} \left( 1 + q(x/\sqrt{n}) \right) \e^{ - n r(x/\sqrt{n} ) },
\end{align}
where $D^n\in(0,\infty)$, $n\in\N$ are normalizing constants.
For a moment it will be useful to consider the unnormalized function $\widetilde{J}^n(x) := (D^n)^{-1} J^n(x)$. Our next lemma states two things. First of all, that in a large interval containing the origin, $\widetilde{J}^n$ is within distance $O(1/\sqrt{n})$ of the standard Gaussian density. It also states that outside of this large interval, $\widetilde{J}^n$ has well behaved tails. All $O( \cdot )$ terms refer to a constant that may depend on $g$ and $h$ but is independent of $x$ and $n$.  

\begin{lemma} \label{lem:density control}
Let $\widetilde{J}^n(x) :=  \frac{ \e^{ - x^2/2} }{ \sqrt{2\pi}} \left( 1 + q(x/\sqrt{n}) \right) \e^{ - n r(x/\sqrt{n} ) }$. Then we have the following three bounds. 
\begin{itemize}
\item For all $|x| \in [0,(\delta \sqrt{n}) \wedge n^{1/6}]$, we have
\begin{align*}
\widetilde{J}^n(x) = \frac{ \e^{ - x^2/2} }{ \sqrt{2\pi}} \left( 1 + O \left( \frac{ |x| + |x|^3}{\sqrt{n} } \right)  \right) .
\end{align*}
\item For all $|x| \in [(\delta \sqrt{n}) \wedge n^{1/6}, \delta \sqrt{n} ]$, we have 
\begin{align} \label{eq:intermediate}
\widetilde{J}^n(x) \leq  \e^{ - \frac{1}{4} n^{1/3} } 
\end{align}
\item For all $|x| \in [\delta \sqrt{n},\infty)$, we have
\begin{align*}
\widetilde{J}^n(x) := O\left(  \left( 1 + |x/\sqrt{n}| \right)^{ \alpha - cn} \right).
\end{align*}
\end{itemize}
\end{lemma} 

\begin{proof}
First we control $\widetilde{J}^n(x)$ for local $x$. With $\delta$ as in \eqref{eq:local bounds} we observe that whenever $|x| \leq  (\delta \sqrt{n}) \wedge n^{1/6}$, we have
\[
|n r( x/\sqrt{n})| \leq \frac{1}{4\delta} \frac{ |x|^3}{ \sqrt{n} } \leq \frac{1}{4 \delta}.
\]
Thus, in particular, $\e^{ - n r(x/\sqrt{n} )} = 1 + O( \frac{|x|^3}{ \sqrt{n} })$ uniformly for $|x| \leq  (\delta \sqrt{n}) \wedge n^{1/6} ]$. Moreover, again by \eqref{eq:local bounds} we clearly have $q(x/\sqrt{n}) = O( |x|/\sqrt{n})$ uniformly for $|x| \leq  (\delta \sqrt{n}) \wedge n^{1/6}$. It follows that uniformly for $|x| \leq  (\delta \sqrt{n}) \wedge n^{1/6} $, we have 
\begin{align*}
\widetilde{J}^n(x) := \frac{ \e^{ - x^2/2} }{ \sqrt{2\pi}} \left( 1 + q(x/\sqrt{n}) \right) \e^{ - n r(x/\sqrt{n} ) } =  \frac{ \e^{ - x^2/2} }{ \sqrt{2\pi}} \left( 1 + O \left( \frac{ |x| + |x|^3}{\sqrt{n} } \right)  \right).
\end{align*}
Next up, we consider intermediate values of $x$, i.e. those $x$ for which $|x| \in (\delta \sqrt{n} \wedge n^{1/6}, \delta \sqrt{n}]$. Again by virtue of \eqref{eq:local bounds} we have $|r(x/\sqrt{n})| \leq \frac{1}{4 \delta } |x/\sqrt{n}|^3 \leq \frac{1}{4} | x/\sqrt{n}|^2$ for $|x/\sqrt{n}| \leq \delta$, so that in particular,
\begin{align} \label{eq:dach}
\frac{1}{2}x^2 + n r(x/\sqrt{n})  \geq \frac{1}{4} x^2 \qquad\text{whenever}\quad |x| \leq \delta \sqrt{n}.
\end{align} 
Moreover, by \eqref{eq:local bounds} we have 
\begin{align} \label{eq:dach2}
|1 + q(x/\sqrt{n})| \leq \frac{5}{4} \qquad\text{whenever}\quad |x| \leq \delta \sqrt{n}.
\end{align}
Combining \eqref{eq:dach} with \eqref{eq:dach2} in the definition of $\tilde{J}_n$, we have 
\begin{align} \label{eq:window}
\tilde{J}_n(x) \leq \frac{5}{4} \frac{1}{ \sqrt{ 2 \pi }} e^{ - x^2/4} \qquad\text{whenever}\quad |x| \leq \delta \sqrt{n}.
\end{align}
In particular, restricting the bound in \eqref{eq:window} to $|x| \geq n^{1/6}$, we obtain
\begin{align*}
\tilde{J}_n(x) \leq \frac{5}{4} \frac{1}{ \sqrt{ 2 \pi }} e^{ - n^{1/3}/4} \qquad\text{whenever}\quad |x| \in (\delta \sqrt{n} \wedge n^{1/6}, \delta \sqrt{n}].
\end{align*}
In particular, since $\frac{5}{4} \frac{1}{ \sqrt{ 2 \pi }} \leq 1$, we obtain \eqref{eq:intermediate}.

Finally, we note that for $|x| \geq  \delta \sqrt{n}$, from \eqref{eq:tail bounds} we have 
\begin{align*}
\widetilde{J}^n(x) \leq C(1 + |x/\sqrt{n}|^\alpha) \e^{ - n c \log(1 + |x/\sqrt{n}| ) } = O \left( \left( 1 + |x/\sqrt{n}| \right)^{ \alpha - cn } \right),
\end{align*}
as required.
\end{proof}

Our next result utilizes Lemma \ref{lem:density control}, which stated that the unnormalized function $\widetilde{J}^n$ was similar to the Gaussian distribution, to control the moments of the probability density $J^n(x) = D^n \widetilde{J}^n(x)$.

\begin{lemma} \label{lem:moment control}
For the normalizing constant $D^n$ of $J^n$, we have 
$D^n = 1 + O(1/\sqrt{n})$, and
\begin{align*}
\mu^n := \int_{-\infty}^\infty x J^n(x) \mathrm{d}x = O(1/\sqrt{n}) \qquad \text{and}  \qquad (\sigma^n)^2 := \int_{-\infty}^\infty x^2 J^n(x) \mathrm{d} x = 1 + O(1/\sqrt n).
\end{align*}
\end{lemma}

\begin{proof}
By the first point in Lemma \ref{lem:density control}, for $k = 0,1,2$ we have
\begin{align} \label{eq:bound1}
\int_{ |x| \in [0,(\delta \sqrt{n} ) \wedge n^{1/6} ]}  x^k \widetilde{J}^n(x)\, \mathrm{d} x = \ind_{ \{0,2\} }(k) + O(1/\sqrt{n} ).
\end{align}
By the second point in Lemma \ref{lem:density control}, for $k = 0,1,2$ we have 
\begin{align} \label{eq:bound2}
\int_{ |x| \in [(\delta \sqrt{n} ) \wedge n^{1/6} ,\delta \sqrt{n} ] } x^k \widetilde{J}^n(x)  \,\mathrm{d} x = O\left( n^{3/2} \e^{ - \frac{1}{4} n^{1/3} } \right) = O(1/\sqrt{n}). 
\end{align}
Finally, by the third point in Lemma \ref{lem:density control} there is a constant 
$C\in(0,\infty)$ independent of $x$ and $n$ such that for $k =0,1,2$ we have 
\begin{align} \label{eq:bound3a}
\int_{ |x| \in [\delta \sqrt{n},\infty) } x^k \widetilde{J}^n(x) \,\mathrm{d} x \leq
\int_{ |x| \in [\delta \sqrt{n},\infty) }  \frac{Cx^k}{|1 + x/\sqrt{n}|^{cn-\alpha} } \,\mathrm{d} x .
\end{align}
By changing variable, we now show that the integral on the right-hand side of \eqref{eq:bound3a} decays exponentially in $n$. Indeed,
\begin{align}
\int_{ |x| \in [\delta \sqrt{n},\infty) }  \frac{Cx^k}{|1 + x/\sqrt{n}|^{cn-\alpha} }\, \mathrm{d} x &\leq 2  C \int_{ \delta \sqrt{n}}^\infty \frac{x^k}{ (1+ x/\sqrt{n})^{ cn - \alpha } } \,\mathrm{d} x \nonumber \\
&\leq 2  C \int_{ \delta \sqrt{n}}^\infty \frac{1}{ (1+ x/\sqrt{n})^{ cn - \alpha - k} } \,\mathrm{d} x \nonumber \\
&= 2 C n^{\frac{k+1}{2} } \int_\delta^\infty \frac{1}{(1 + y)^{cn - \alpha - k}} \,\mathrm{d}y \nonumber \\
&= \frac{2 C n^{\frac{k+1}{2} }}{ (c n - \alpha - k + 1) (1+ \delta)^{c n - \alpha - k + 1 ) }}, \label{eq:bound3b}
\end{align}
which decays exponentially as $n$ increases. In particular, combining \eqref{eq:bound3a} and \eqref{eq:bound3b} we obtain 
\begin{align} \label{eq:bound3}
\int_{ |x| \in [\delta \sqrt{n},\infty) } x^k \widetilde{J}^n(x) \,\mathrm{d} x = O(1/\sqrt{n}).
\end{align}
It now follows from setting $k=0$ in \eqref{eq:bound1}, \eqref{eq:bound2} and \eqref{eq:bound3} that 
\begin{align} \label{eq:bound4}
D^n := \left( \int_{-\infty}^\infty \widetilde{J}^n(x) \,\mathrm{d} x \right)^{ - 1} = 1 + O(1/\sqrt{n}).
\end{align} 
The claimed facts about the mean and variance of $J^n$ follow from combining \eqref{eq:bound4} with setting $k=1$ and $k=2$ in \eqref{eq:bound1}, \eqref{eq:bound2} and \eqref{eq:bound3}.
\end{proof}

So far, in Lemmas \ref{lem:density control} and \ref{lem:moment control} we have seen that roughly speaking $J^n$ is within $O(1/\sqrt{n})$ of the standard Gaussian density. In the following, we will consider a corrected version of $J^n$ to have zero mean and unit variance. Indeed, with $\mu^n$ and $\sigma^n$ as in Lemma \ref{lem:moment control} we have, for all $x\in\R$,
\begin{align} \label{eq:I def}
I^n(x) := \sigma^n J^n( \mu^n + \sigma^n x).
\end{align}
It is plain from the definition that for $k =0,1,2$
\begin{align} \label{eq:first three}
\int_{-\infty}^\infty x^k I^n(x)\, \mathrm{d} x = \ind_{ \{0,2\} }(k).
\end{align}

Our next lemma is essentially an analogue of Lemma \ref{lem:density control} for $I^n$ rather than $J^n$, stating that $I^n$ is close to the Gaussian density on a large interval containing the origin, and that $I^n$ has well behaved tails.

\begin{lemma} \label{lem:I control}
We have the following three bounds.

\begin{itemize}
\item 
For all $|x| \in [0, (\delta \sqrt{n}) \wedge n^{1/6}-1]$, we have
\begin{align*}
I^n(x) = \frac{\e^{ - x^2/2} }{ \sqrt{2 \pi }} \left( 1 + O \left( \frac{1 + |x|^3}{ \sqrt{n} } \right) \right).
\end{align*}
\item
For all $|x| \in [(\delta \sqrt{n}) \wedge n^{1/6}-1, \delta \sqrt{n} - 1 ]$, we have
\begin{align*}
I^n(x) = O\left( \e^{ - \frac{1}{4} n^{ 1/3} } \right) .
\end{align*}
\item
For all $|x| \in [ \delta \sqrt{n} - 1 , \infty) $, we have
\begin{align*}
I^n(x) = O \left( \frac{1}{ ( 1 + |x/\sqrt{n}| )^{cn - \alpha} } \right).
\end{align*} 
\end{itemize}
\end{lemma}

\begin{proof}
To prove the first point, note that $(\sigma^n)^2 = 1 + O(1/\sqrt{n})$ implies $\sigma^n = 1 + O(1/\sqrt{n})$. Moreover, since $D^n = 1 + O(1/\sqrt{n})$ also, and $\mu^n = O(1/\sqrt{n})$, by virtue of the first point in Lemma \ref{lem:density control} for all $|x| \in [0, n^{1/6} - 1] \subset [0,n^{1/6} - O(1/\sqrt{n})]$, we have 
\begin{align*}
I^n(x) := \sigma^n D^n \widetilde{J}^n( \mu^n + \sigma^n x ) = \frac{ \e^{ - x^2/2}}{ \sqrt{2\pi}} \left( 1 + O \left( \frac{1 + |x|^3}{ \sqrt{n} } \right) \right).
\end{align*}
As for the second point, we note that since for sufficiently large $n$, we have 
$$
[n^{1/6} -1, \sqrt{n}-1] \subset [n^{1/6}/2 + O(n^{1/6}/\sqrt{n}) , \sqrt{n} - O(O(n^{1/6}/\sqrt{n}) ],
$$ 
we may use $I^n(x)  =  \sigma^n D^n \widetilde{J}^n( \mu^n + \sigma^n x )$ in conjunction with the second part of Lemma \ref{lem:density control}.

Finally the third claim follows quickly from the third claim of Lemma \ref{lem:density control} since
\begin{align*}
O \left( \frac{1}{ (1 + |\mu^n + \sigma^n x / \sqrt{n} | )^{cn - \alpha} } \right) = O  \left( \frac{1}{ (1 + | x / \sqrt{n} | )^{cn - \alpha} } \right).
\end{align*}
\end{proof}

To recapitulate on our work in this section, we have shown that if $X^n$ is a random variable distributed according to $\rho^n$ as in Equation \eqref{eq:rho def}, where $(g,h)$ is an  admissible pair, then the normalized variable
\begin{align*}
\widetilde{X}^n := \frac{X^n - \mathbb{E}[X^n ]}{ \sqrt{ \mathrm{Var}[X^n] } }
\end{align*}
is distributed according to a probability density $I^n$ that has zero mean and unit variance, and is close to the standard Gaussian density in the sense that Lemma \ref{lem:I control} holds.

In the next section we utilize the similarity of $I^n$ with the Gaussian density in order to show that the characteristic function of $\widetilde{X}^n$ is similar to that of the standard Gaussian density. 

\subsection{Characteristic functions} 
Our next lemma states that when $n$ is large, the characteristic function of $\widetilde{X}^n$ is similar to that of the standard Gaussian density.

\begin{lemma} \label{lem:laplace char}
Recall from \eqref{eq:I def} that $I^n$ is a rescaling of $\rho^n$ that has zero mean and unit variance. Let $\varphi^n$ be the associated characteristic function, that is
\begin{align*}
\varphi^n(t) := \int_{ -\infty}^\infty \e^{ i t x } I^n(x) \,\mathrm{d} x.
\end{align*}
Then for a constant $C_{g,h}$ independent of $n$ and $t$ we have 
\begin{align} \label{eq:laplace char}
\left| \varphi^n(t) - \e^{ - t^2/2} \right| \leq \frac{C_{g,h}}{\sqrt{n}} |t|^3.
\end{align}
\end{lemma}

\begin{proof}
Expressing $\e^{ -t^2/2}$ as an integral, we have  
\begin{align*}
\left| \varphi^n(t) - \e^{ - t^2/2} \right| = \left| \int_{ -\infty}^\infty \e^{ it x} \left( I^n(x) - \frac{ \e^{ -x^2}}{ \sqrt{2\pi } }  \right) \mathrm{d}x \right|.
\end{align*}
Now by Taylor's expansion, there is a function $\theta:\mathbb{R} \to \mathbb{C}$ satisfying $|\theta(u)| \leq 1$ for all $u \in \mathbb{R}$ such that $\e^{ it x} = 1 - i t x - t^2 x^2/2 - \frac{ i |x|^3 |t|^3 \theta(t x) }{6}$ for all $t,x\in\R$. In particular, since the mean and variance of $I^n(x)$ agree with that of the standard Gaussian density, i.e., \eqref{eq:first three} holds, we have 
\begin{align} \label{eq:three int}
\left| \varphi^n(t) - \e^{ - t^2/2} \right| \leq  \frac{ |t|^3}{ 6} \int_{ -\infty}^\infty |x|^3  \left| I^n(x) - \frac{ \e^{ -x^2}}{ \sqrt{2\pi } } \right|   \mathrm{d} x.
\end{align}
Note that by virtue of Lemma \ref{lem:I control}, we obtain the following.
\begin{itemize}
\item For all $|x| \in [0,n^{1/6}-1]$, we have 
\begin{align*}
 I^n(x) - \frac{ \e^{ -x^2}}{ \sqrt{2\pi } } = O \left( \frac{1 + |x|^3}{ \sqrt{n} } \right) \e^{ - x^2/2}.
\end{align*}
\item For all $|x| \in [n^{1/6}-1,\sqrt{n} - 1]$, we have 
\begin{align*}
 I^n(x) - \frac{ \e^{ -x^2}}{ \sqrt{2\pi } } = O \left( \e^{ - \frac{1}{16} n^{1/3} } + \e^{ - \frac{1}{2} (n^{1/6})^2 }  \right).
\end{align*}
\item Finally, for all $|x| \in [\sqrt{n}-1,\infty)$, we have 
\begin{align*}
 I^n(x) - \frac{ \e^{ -x^2}}{ \sqrt{2\pi } } =O \left( \frac{1}{ ( 1 + |x/\sqrt{n}| )^{cn - \alpha} } + \e^{ - \frac{1}{2} x^2} \right).
\end{align*}
\end{itemize}
These three bounds may be applied to control the integrand in \eqref{eq:three int}, so that it is straightforward to show that 
\begin{align} 
\left| \varphi^n(t) - \e^{ - t^2/2} \right| \leq C |t|^3/ \sqrt{n}
\end{align}
for a constant $C\in(0,\infty)$ depending on $g$ and $h$ as they appear in $\rho^n$ (and implicitly in $I^n$), but independent of $n$.

\end{proof}

While Lemma \ref{lem:laplace char} was concerned with the characteristic function of $\widetilde{X}_n$, in our next lemma we look at the characteristic function of the normalized sum
\begin{align*}
\widetilde{S}^n := \frac{ \sum_{ i = 1}^p \widetilde{X}^n_i }{ \sqrt{p} },
\end{align*}
where $X^n_1,\ldots,X^n_p$ are independent and identically distributed according to probability density $I^n$. Roughly speaking, where Lemma \ref{lem:laplace char} stated that the characteristic function of $\widetilde{X}^n$ was within $O(1/\sqrt{n})$ of the Gaussian characteristic function, our next result states that this bound improves to $O(1/\sqrt{pn})$ when considering a normalized sum of $p$ copies.

\begin{lemma} \label{lem:laplace char higher}
Whenever $p \geq 2$, $n \geq 64 \e^4 C^2$, and $|t| \leq 2 \sqrt{p}$, we have
\begin{align*}
\left| \varphi^n(t/\sqrt{p})^p - \e^{ -t^2/2}  \right| \leq \frac{ C_{g,h} }{ \sqrt{np}} |t|^3 \exp \left \{ - t^2 /8 \right\},
\end{align*}
where $C_{g,h}\in(0,\infty)$ is as in Lemma \ref{lem:laplace char}.
\end{lemma}

\begin{proof}
Note that whenever $u,v$ are complex numbers such $|u|,|v| \leq a$, we have the bound $|u^p - v^p | \leq p | u - v| a^{p-1}$. With this inequality in mind, let $u = \varphi^n(t/\sqrt{p})$ and $v := \e^{ -t^2/2p}$. Then using Lemma \ref{lem:laplace char} both $u$ and $v$ are bounded in modulus by $a := \e^{ - t^2/2p} + C \frac{ |t|^3}{ \sqrt{n} p^{3/2} }$. In particular, again using Lemma \ref{lem:laplace char} to bound $|u - v|$, we have 
\begin{align*}
\left| \varphi^n(t/\sqrt{p})^p - \e^{ -t^2/2}  \right| &\leq \frac{  C}{  \sqrt{n p }} |t|^3 \left( \e^{ - \frac{t^2}{ 2p } } + \frac{C}{\sqrt{n p^3} } |t|^3 \right)^{ p- 1} \\
&=  \frac{C}{ \sqrt{n p} } |t|^3 \e^{ - \frac{p-1}{p} \frac{t^2}{2} }   \left( 1 + \frac{C}{ \sqrt{n p^3 }} |t|^3 \e^{t^2/2p} \right)^{ p - 1} \\
&\leq  \frac{C}{ \sqrt{n p} } |t|^3 \exp \left(  - \frac{p-1}{p} \frac{t^2}{2} \left( 1 - 2 \frac{C}{ \sqrt{n p}} |t| \e^{t^2/2p} \right) \right).
\end{align*}  
Now provided $|t| \leq 2 \sqrt{p}$, we may bound the internal term in the exponent, so that 
\begin{align*}
\left| \varphi^n(t/\sqrt{p})^p - \e^{ -t^2/2}  \right| &\leq  \frac{C}{ \sqrt{n p} } |t|^3 \exp \left(  - \frac{p-1}{p} \frac{t^2}{2} \left( 1 - 4 \e^2 \frac{C}{ \sqrt{n}}  \right) \right).
\end{align*}  
Provided $n \geq 64 \e^4 C^2$, $\left( 1 - 4 \e^2 \frac{C}{ \sqrt{n}}  \right)  \geq 1/2$. Moreover, whenever $p \geq 2$, $\frac{p-1}{p} \geq 1/2$, so that under these conditions 
\begin{align*}
\left| \varphi^n(t/\sqrt{p})^p - \e^{ -t^2/2}  \right| &\leq  \frac{C}{ \sqrt{n p} } |t|^3 \e^{ - t^2/8},
\end{align*}  
as required.
\end{proof}

We will ultimately like to use the Berry smoothing inequality to show that $\widetilde{S}^n$ is within $O(1/\sqrt{pn})$ of a standard Gaussian random variable. To this end, we need control over the characteristic function $\varphi^n(t/\sqrt{p})^p$ of $\widetilde{S}^n$ in a region of size order $\sqrt{np}$. Lemma \ref{lem:laplace char higher} only provides coverage up in a region of size $\sqrt{p}$. Our next lemma supplies a tail bound taking care of the region outside of $\sqrt{p}$.  

\begin{lemma} \label{lem:global}
For all $t\in\R\setminus\{0\}$, 
\begin{align*}
\left| \varphi^n(t/\sqrt{p})^p  \right| \leq \frac{1}{ |t/\sqrt{p}|^p} .
\end{align*}
\end{lemma}

\begin{proof}
Whenever a density function $f$ is differentiable on $\R$, it is easily verified by integration by parts that 
\begin{align*}
\left| \int_{ -\infty}^\infty \e^{ it x } f(x) \mathrm{d} x \right| \leq \frac{1}{|t|} \int_{ -\infty}^\infty |f'(x)| \mathrm{d} x . 
\end{align*}
Now since $\rho_n$ has a unique maximum, so does the normalized density $I^n$, and since $D^n = 1 + O(1/\sqrt{n})$ this maximum has takes the form $\frac{1}{ \sqrt{2 \pi }} + O(1/ \sqrt{n})$. In particular, there exists some $n_0\in\N$ such that for all $n \geq n_0$, we have
\begin{align*}
\sup_{ x \in \mathbb{R} } I^n(x) \leq \frac{1}{2}.
\end{align*}
Moreover, by assumption $(a)$ of Definition \ref{df:admissible} we have 
\begin{align*}
\int_{-\infty}^\infty | (I^n)'(x) | \mathrm{d} x = 2 \sup_{ x \in \mathbb{R}} H_n(x) \leq 1.
\end{align*}
In particular, the characteristic function $\varphi^n$ satisfies the inequality
\begin{align*}
| \varphi^n( t)| \leq 1/|t|.
\end{align*}
for all $t\in\R\setminus\{0\}$. The result for $\varphi^n(t/\sqrt{p})^p$ follows. 
\end{proof}

In the next section we complete the proof of Theorem \ref{thm:laplace berry}. 

\subsection{Proof of Theorem \ref{thm:laplace berry}}
We now prove Theorem \ref{thm:laplace berry}.

\begin{proof}[Proof of Theorem \ref{thm:laplace berry}]
By setting $T = \infty$ in the Berry smoothing inequality (see, e.g., \cite[Section 7.4]{chung2001course}), the Kolmogorov-Smirnov distance between $\widetilde{S}^{n}_p$ and a standard Gaussian random variable $G$ may be bounded via
\begin{align*}
\mathrm{d}_{KS}\left( \widetilde{S}^n_p , G \right)\leq  \frac{1}{\pi}  \int_{-\infty}^\infty \frac{ |\varphi^n(t/\sqrt{p})^p  - \e^{ -t^2/2} | }{ |t|} \,\mathrm{d}t .
\end{align*}
Using Lemmas \ref{lem:laplace char higher} and \ref{lem:global} to respectively control the integrand inside and outside of $[-2 \sqrt{p}, 2 \sqrt{p}]$, we have 
\begin{align*} 
\mathrm{d}_{KS}\left( \widetilde{S}^n_p , G \right)\leq \frac{1}{ \pi} \frac{C}{\sqrt{np}}  \int_{-2\sqrt{p}}^{2 \sqrt{p} } |t|^3 \e^{ - t^2/8} \frac{ \mathrm{d}t }{ |t|} + \int_{|t| > 2 \sqrt{p} } \left( \e^{ -t^2/2} + \frac{1}{ |t/\sqrt{p}|^p } \right) \frac{ \mathrm{d}t}{ |t|}.
\end{align*}
Performing each of the integrals, we find that there is a constant $C\in(0,\infty)$ independent of $n$ and $p$ such that 
\begin{align*}
\mathrm{d}_{KS}\left( \widetilde{S}^n_p , G \right) \leq C \left( \frac{1}{ \sqrt{np}} + 2^{ - p} \right),
\end{align*}
completing the proof.
\end{proof}

\section{Proof of Theorem \ref{thm:main}} \label{sec:main}

In this section we prove Theorem \ref{thm:main}. Let $W_{n,p}^{G,H}$ be the log-volume of a random simplex whose vertices $Y_1^p,\ldots,Y_p^n$ are independent and identically distributed according to $\mu^n$ as in \eqref{eq:mu def}. Then, by the distributional equality \eqref{eq:main rep}, 
\begin{align} \label{eq:chaos}
W_{n,p}^{G,H} \stackrel{\dint}{=} W_{n,p}^{\mathrm{Sph}} + \sum_{ j = 1}^p \log R_j^n,
\end{align} 
where $\log R_j^n$ are independent and identically distributed with the law $\log \twonorm{Y^n}$, where $Y^n \sim \mu^n$. The proof of Theorem \ref{thm:main} hinges on the idea that both terms on the right-hand side of \eqref{eq:chaos} are close in distribution to a standard Gaussian random variable, and these facts may be synthesized by the following parallelogram inequality for Kolmogorov-Smirnov distances.

\begin{lemma}
Let $X,X',Y,Y'$ be independent real-valued random variables. Then
\begin{align} \label{quad}
\mathrm{d}_{KS} \left( X+Y , X' + Y' \right) \leq \mathrm{d}_{KS} \left( X , X'  \right) + \mathrm{d}_{KS} \left( Y, Y' \right).
\end{align} 
\end{lemma}
\begin{proof}
It is immediate from the definition that Kolmogorov-Smirnov distances satisfy the triangle inequality. That is, if $\mathrm{d}_{KS} (A,B) := \sup_{s \in \mathbb{R} } | \mathbb{P} (A < s ) - \mathbb{P} ( B  < s ) |$, then 
\begin{align}  \label{triangle space}
\mathrm{d}_{KS}(A,C) \leq \mathrm{d}_{KS}(A,B) + \mathrm{d}_{KS}(B,C).
\end{align}
On the other hand, write $f(s) := | \mathbb{P}( X < s ) - \mathbb{P}( X' < s )|$. Then $\left| \mathbb{P}( X + Y < s ) - \mathbb{P}( X' + Y < s ) \right| = \mathbb{E}[ f( s - Y) ] \leq \sup_{s \in \mathbb{R}} f(s)$. It follows in particular that
\begin{align} \label{parallel}
\dKS(X+Y,X'+Y)\le \dKS (X,X').
\end{align}
The inequality \eqref{quad} may be proved by letting $A = X + Y$, $B = X' + Y$ and $C = X' + Y'$, and subsequently using \eqref{triangle space} followed by \eqref{parallel}.
\end{proof}
Specializing to distances from Gaussian random variables, we have the following corollary.

\begin{cor} \label{cor:gauss para}
Let $X, Y$ be independent random variables with zero mean and unit variance. Then for real numbers $\sigma, \tau$ (not zero simultaneously) and $N$ a standard Gaussian, we have 
\begin{align*}
\mathrm{d}_{KS} \left( \frac{ \sigma X + \tau Y}{ \sqrt{ \sigma^2 + \tau^2 } } , N \right) \leq \mathrm{d}_{KS}(X,N) + \mathrm{d}_{KS}(Y,N).
\end{align*} 
\end{cor}

We are now ready to prove Theorem \ref{thm:main}.

\begin{proof}[Proof of Theorem \ref{thm:main}]
We will show that when $p$ and $n$ are large, both terms on the right-hand-side of \eqref{eq:chaos} are close in distribution to standard Gaussian random variables. Indeed, considering the sum over $j$ first, by using the polar integration formula, it follows that for $r > 0$ we have 
\begin{align*}
\mathbb{P} \left( \twonorm{Y_1^n} \in \mathrm{d} r \right) = \widetilde{C}^nn r^{n-1} G(r) \e^{ - n H(r) } \mathrm{d}r.
\end{align*} 
for some constant $\widetilde{C}^n\in(0,\infty)$. Transforming, it is verified that $\log \twonorm{Y_1^n}$ is then distributed according to the probability measure on $\mathbb{R}$ whose density function is given by 
\begin{align*}
\rho^n(r) := \widetilde{C}^n g(r) \e^{ - n h(r)},\qquad r\in \R, 
\end{align*}
where we recall that $g(r) = G(\e^r)$ and $h(r) = H(\e^r) -r$. 

In particular, since $(G,H)$ are radially admissible, i.e., $(g,h)$ are admissible, so that Theorem \ref{thm:laplace berry} applies. In particular, there is a constant $C_{G,H}\in(0,\infty)$ and $n_0\in\N$ depending on $(G,H)$ such that for all $n \geq n_0$ we have
\begin{align} \label{eq:oranges}
\mathrm{d}_{KS} \left( \frac{ \sum_{j = 1}^p \log R_j^n -p  \mathbb{E}[ \log R_1^n ]}{ \sqrt{ \mathrm{Var}[ \log R_1^n ] } } , N \right) \leq C_{G,H} \left( \frac{1}{ \sqrt{ p n } }+ 2^{ - p} \right).
\end{align}
On the other hand, using Theorem \ref{thm:spherical berry} we have 
\begin{align} \label{eq:apples}
\mathrm{d}_{KS} \left( \frac{ W_{n,p}^{\mathrm{Sph}} - \mathbb{E} [ W_{n,p}^{\mathrm{Sph}} ] }{ \sqrt{ \mathrm{Var}[ W_{n,p}^{\mathrm{Sph} } ] } }  , N \right) \leq   \frac{ C ~ \theta^2  }{ n(1-\theta) \left[ \log \frac{1}{1-\theta} - \theta \right]^{3/2}  }. 
\end{align}
Combining \eqref{eq:oranges} and \eqref{eq:apples}, \eqref{eq:chaos} and making use of Corollary \ref{cor:gauss para}, we obtain
\begin{align} \label{eq:pears}
\mathrm{d}_{KS} \left( \frac{ W_{n,p}^{G,H} - \mathbb{E} [ W_{n,p}^{G,H} ] }{ \sqrt{ \mathrm{Var}[ W_{n,p}^{G,H } ] } }  , G \right) \leq  C_{G,H} \left( \frac{1}{ \sqrt{ p n } }+ 2^{ - p} \right) +    \frac{ C ~ \theta^2  }{ n(1-\theta) \left[ \log \frac{1}{1-\theta} - \theta \right]^{3/2}  }. 
\end{align}
The result follows from the observation that the former bound is finer than the latter. That is, for all $1 \leq p \leq n$, there is a constant $C\in(0,\infty)$ such that with $\theta = \frac{p-1}{n}$, we have 
\begin{align*}
 \frac{1}{ \sqrt{ pn} }+ 2^{ - p}  \leq \frac{2}{p} \leq C \frac{ \theta^2  }{ n(1-\theta) \left[ \log \frac{1}{1-\theta} - \theta \right]^{3/2} }
\end{align*} 
for some constant $C\in(0,\infty)$. That completes the proof of Theorem \ref{thm:main}. 
\end{proof}

\section{Proof of Theorem \ref{thm:gaussian matrix}} \label{sec:gaussian proof}

In this section we provide a direct proof of Theorem \ref{thm:gaussian matrix}, which states that if $A^{n}$ is an $n \times n$ matrix with standard Gaussian entries, then
\begin{align*}
\mathrm{d}_{KS} \left( \frac{ \log |\det (A^n )| - ( \frac{1}{2} \log (n-1)! + c_0 ) }{ \sqrt{ \frac{1}{2} \log n + c_1 } } , N \right)  \leq \frac{ C }{ \log^{3/2} n },
\end{align*}
With the exception of a few definitions and bounds relating to the polygamma functions that we import from Section \ref{sec:polygamma}, this section is independent of the remainder of the paper, though several parts run closely in parallel with ideas seen in Section \ref{sec:spherical}.

Now let $A^n$ be an $n \times n$ matrix whose entries are independent standard Gaussian random variables. The starting point of our analysis is the well known identity in law 
\begin{align} \label{eq:gauss identity}
|\mathrm{det} (A_n)| \stackrel{\dint}{=} 2^{n/2} \left( \prod_{j=1}^n R_{j/2} \right)^{1/2},
\end{align}
dating back to Goodman \cite{goodman1963distribution}, where $R_{1/2},\ldots,R_{n/2}$ are independent random variables such that $R_{j/2}$ has the Gamma distribution with shape parameter $j/2$ and unit scale parameter. 

Taking logarithms of \eqref{eq:gauss identity}, we may express the log-determinant of $|\mathrm{det}(A_n)|$ in terms of an independent sum of log-gamma random variables: 
\begin{align} \label{eq:gauss identity log}
\log |\mathrm{det} (A_n) | \stackrel{\dint}{=} \frac{n}{2} \log 2 + \frac{1}{2} \sum_{j=1}^n \log R_{j/2} .
\end{align}
A brief calculation tells us that if $\mathbb{P} \left[ W \in \mathrm{d}\zeta \right] = \frac{ 1}{ \Gamma(\lambda)} \zeta^{\lambda - 1} \e^{ - \zeta} \mathrm{d} \zeta$ tells us that
\begin{align*}
\mathbb{E}[ W ] = \psi_0( \lambda ) \qquad \text{and} \qquad \mathrm{Var}[ \log W ] = \psi_1(\lambda),
\end{align*}
so that in particular 
\begin{align*}
\mathbb{E} \left[ \log |\mathrm{det} (A_n)| \right] = \frac{n}{2} \log 2 + \frac{1}{2} \sum_{ j = 1}^n \psi_0( j/2) 
\end{align*}
and 
\begin{align*}
\mathrm{Var} \left[ \log |\mathrm{det} (A_n)| \right] = \frac{1}{4} \sum_{ j = 1}^n \psi_1( j/2).
\end{align*}

We now compute the characteristic function of a normalized log-gamma random variable. The following lemma is an analogue of Lemma \ref{lem:key moments} with the log-gamma random variable in place of the log-beta random variable. Since the proof is rather similar --- and simpler --- we will be content to sketch just a few key details.

\begin{lemma}  \label{lem:log gamma char}
Let $W_\lambda$ be gamma distributed with parameter $\lambda>0$, and define
\begin{align*}
\widetilde{Q}_\lambda := \frac{ \log W_\lambda - \psi_0(\lambda) }{ \sqrt{ \psi_1(\lambda) } } .
\end{align*}
Then, for all $t\in\R$,
\begin{align*}
\varphi_\lambda( t) := \mathbb{E} [ \e^{ it Q_\lambda} ] = \exp \left( - t^2/2 - i \int_0^{t/\sqrt{\psi_1(\beta)}} \mathrm{d}t_1 \int_0^{t_1} \mathrm{d}t_2 \int_0^{t_2} \mathrm{d} t_3 \psi_2 \left( \beta + i t_3 \right) \right).
\end{align*}
\end{lemma}

\begin{proof}
A basic calculation tells us that
\begin{align*}
\mathbb{E} [ \e^{ i t \log W_\lambda} ] = \frac{ \Gamma( \lambda + it) }{ \Gamma( \lambda) } = \exp \left(  i \int_0^t \psi_0( \lambda + i t_1) \mathrm{d}t_1 \right).
\end{align*} 
In particular
\begin{align*}
\mathbb{E} [ \e^{ i t \widetilde{Q}_\lambda} ]= \exp \left\{   -  i t \frac{ \psi_0(\lambda)}{ \sqrt{ \psi_1(\lambda) } } +  i \int_0^{t/\sqrt{ \psi_1(\lambda)} }  \psi_0( \lambda + i t_1) \mathrm{d}t_1 \right\} =  \exp \left\{   - \int_0^{t/\sqrt{ \psi_1(\lambda)} } \mathrm{d}t_1 \int_0^{t_1} \mathrm{d}t_2 ~  \psi_1( \lambda + i t_2)  \right\}.
\end{align*} 
Pulling out a factor of $t^2/2$ from the integrand to obtain the first inequality below, and packaging the difference as in integral to obtain the second, we have  
\begin{align*}
\mathbb{E} [ \e^{ i t \widetilde{Q}_\lambda} ]&= \exp \left\{   - t^2/2 -  \int_0^{t/\sqrt{ \psi_1(\lambda)} } \mathrm{d}t_1 \int_0^{t_1} \mathrm{d}t_2 ~  \left( \psi_1( \lambda + i t_2) - \psi_1(\lambda) \right)  \right\}\\
&= \exp \left\{   - t^2/2 -i  \int_0^{t/\sqrt{ \psi_1(\lambda)} } \mathrm{d}t_1 \int_0^{t_1} \mathrm{d}t_2 \int_0^{t_2} \mathrm{d} t_3 ~   \psi_2( \lambda + i t_3) \right\},
\end{align*} 
completing the proof of Lemma \ref{lem:log gamma char}.

\end{proof}

We now note that if $\phi_n(t)$ is the characteristic function of the centering of $\log \mathrm{det}(A^n)$, then using \eqref{eq:gauss identity log} we have 
\begin{align} \label{eq:product rep 1}
\phi_n(t) := \mathbb{E} \left[ \exp \left\{   i t \frac{ \log \mathrm{det}(A^n) - \mathbb{E}[ \log |\mathrm{det} (A^n)| ] }{ \sqrt{ \mathrm{Var}[ \log |\mathrm{det} (A^n)| ] } } \right\} \right] = \prod_{ j = 1}^n \varphi_{j/2} \left( \frac{ \sqrt{ \psi_1( j/2)} }{ S_n  } t \right),
\end{align} 
where $\varphi_{j/2}$ are defined as in Lemma \ref{lem:log gamma char} and $S_n := \sqrt{ \sum_{ k = 1}^n \psi_1(k/2)}$.

Our next lemma expresses how similar the centering of $\log \mathrm{det}(A^n)$ is to a standard Gaussian random variable.

\begin{lemma} \label{lem:gauss char}
For all $t\in\R$, we have
\begin{align*}
\left| \log \phi_n(t) + t^2/2 \right| \leq \varepsilon_n |t|^3
\end{align*}
where
\begin{align*}
\varepsilon_n := 4/ \log^{3/2}n.
\end{align*}
\end{lemma}
\begin{proof}
Using \eqref{eq:product rep 1} in conjunction with Lemma \ref{lem:log gamma char} we have
\begin{align*}
 \log \phi_n(t) + t^2/2 = - i  \int_0^{ t/ S_n } \mathrm{d}t_1 \int_0^{t_1} \mathrm{d}t_2 \int_0^{t_2} \mathrm{d}t_3 \sum_{ j = 1}^n \psi_2 \left( \frac{j}{2} + it_3 \right).
\end{align*}
Using \eqref{eq:polygamma bound 2}, and the fact that the simplex $\{ (t_1,t_2,t_3) \in \mathbb{R}^3 : 0 < t_3 < t_2 < t_1 < a \}$ has volume $a^3/6$, we have 
\begin{align*}
\left|  \log \phi_n(t) + t^2/2 \right| \leq \frac{1}{6} \frac{ |t|^3}{ S_n^3}  \sum_{ j =1}^n \frac{ 8}{ (j/2)^2 } \leq \frac{32}{6}  \frac{ |t|^3}{ S_n^3}\frac{ \pi^2}{ 6} .
\end{align*}
Now using the lower bound in Lemma \ref{eq:polygamma bound} we have $S_n^2 := \sum_{ k = 1}^n \psi_1(k/2) \geq \sum_{ k=1}^n 2/k \geq 2 \log n$. Using the fact that $\frac{32 \pi^2 }{36} \frac{1}{ 2^{3/2} } \leq 4$, the result follows.
\end{proof}

We are now equipped to prove a version of Theorem \ref{thm:gaussian matrix} with implicit means and variances.

\begin{thm} \label{thm:gaussian implicit}
Let $A^n$ be an $n \times n$ matrix with independent standard Gaussian entries. Then
\begin{align*}
d_{KS} \left( \frac{ \log \mathrm{det}(A^n) - \mathbb{E} [ \log \mathrm{det}(A^n) ] }{ \sqrt{ \mathrm{Var}[\log \mathrm{det}(A^n) ] } } , N \right) \leq C/\log^{3/2}n.
\end{align*}
\end{thm}

\begin{proof}
With $\varepsilon$ as in Lemma \ref{lem:gauss char}, using Lemma \ref{lem:berry technique} we see that, for all $|t| \leq (4 \varepsilon_n)^{-1}$, we have $|\phi_n(t) - \e^{ - t^2/2} | \leq \varepsilon |t|^3 \e^{ - t^2/4} $. Now as in the proof of Theorem \ref{thm:spherical berry} in Section \ref{sec:spherical proof}, Theorem \ref{thm:gaussian implicit} follows from Berry's smoothing inequality \eqref{eq:berry}.
\end{proof}

In order to complete the proof of Theorem \ref{thm:gaussian matrix}, we require fine estimates on the mean and variance of $\log \mathrm{det}(A^n)$. To this end we have the following lemma.

\begin{lemma} \label{lem:gaussian moments}
Let $c_0$ and $c_1$ be as in \eqref{eq:c0 def} and \eqref{eq:c1 def}. Then we have
\begin{align} \label{eq:meanguess}
\mathbb{E}[ \log \mathrm{det}(A^n) ] = \frac{1}{2} \log (n-1)! + c_0 + \varepsilon_n
\end{align}
and
\begin{align} \label{eq:varguess}
\mathrm{Var}[ \log \mathrm{det}(A^n) ] = \frac{1}{2} \log n + c_1 + \delta_n,
\end{align}
where for a universal constant $C\in(0,\infty)$, we have $|\varepsilon_n|,|\delta_n| < C/n$. 
\end{lemma}

\begin{proof}
Recall that 
\begin{align} \label{eq:proper rep}
\mathbb{E} \left[ \log \mathrm{det} (A_n) \right] = \frac{n}{2} \log 2 + \frac{1}{2} \sum_{ j = 1}^n \psi_0( j/2) .
\end{align}
We begin with the integral formula 
\begin{align} \label{eq:WW}
\psi_0(z) := \log z - \frac{1}{ 2 z} + \int_0^\infty \left( \frac{1}{2} - \frac{1}{\zeta} + \frac{1}{\e^\zeta - 1} \right) \e^{ - z \zeta} \mathrm{d}\zeta,
\end{align}
(see, e.g., Whittaker and Watson \cite[Section 12.31]{whittaker1965course}). In particular, we may write
\begin{align} \label{eq:first exp}
\sum_{ j = 1}^n \psi_0(j/2) = - n \log 2 + \sum_{ j = 1}^n \log j - \sum_{ j = 1}^n 1/j + \sum_{ j = 1}^n p_j, 
\end{align} 
where
\begin{align*}
p_j := \int_0^\infty \left( \frac{1}{2} - \frac{1}{ \zeta} + \frac{1}{ \e^\zeta - 1} \right) \e^{ - j \zeta / 2} \mathrm{ d} \zeta.
\end{align*}
It is easily verified that
\begin{align*}
\sum_{ j = 1}^\infty p_j =   \int_0^\infty \left( \frac{1}{2} - \frac{1}{ \zeta} + \frac{1}{ \e^\zeta - 1} \right) \frac{1}{ \e^{\zeta/2} - 1 }  \mathrm{ d} \zeta = c_0',
\end{align*} 
and that moreover there is a universal constant $C\in(0,\infty)$ such that 
\begin{align*}
\left| \sum_{ j = n+1 }^\infty p_j \right| = \left|  \int_0^\infty \left( \frac{1}{2} - \frac{1}{ \zeta} + \frac{1}{ \e^\zeta - 1} \right) \frac{\e^{ - n \zeta / 2} }{ \e^{\zeta/2} - 1 }  \mathrm{ d} \zeta \right| \leq C/n,
\end{align*} 
so that in particular
\begin{align*}
\sum_{ j = 1}^n p_j = c_0' - O(1/n).
\end{align*} 
The first equation for the mean now follows from \eqref{eq:proper rep} and \eqref{eq:first exp} in conjunction with the well known bound
\begin{align*}
\sum_{ j = 1}^n 1/j = \log n + \gamma + O(1/n).
\end{align*} 
We turn to the proof of \eqref{eq:varguess}, which is similar. Recall first that 
\begin{align*}
\mathrm{Var} \left[ \log \mathrm{det} (A_n) \right] = \frac{1}{4} \sum_{ j = 1}^n \psi_1( j/2).
\end{align*}
Differentiating through \eqref{eq:WW} and using the identity $\frac{1}{z^2} = \int_0^\infty t \e^{ - zt} \mathrm{d}t$, we have 
\begin{align*}
\psi_1(z) = \frac{1}{z} +  \int_0^\infty \frac{ (\zeta-1) \e^\zeta + 1}{ \e^\zeta - 1 } \e^{ - z \zeta} \mathrm{d}\zeta .
\end{align*} 
In particular,
\begin{align*}
\mathrm{Var} \left[ \log \mathrm{det} (A_n) \right] = \frac{1}{2} \sum_{ j=1}^n 1/j + \frac{1}{4} \sum_{ j = 1}^n s_j,
\end{align*}
where 
\begin{align*}
s_j :=   \int_0^\infty \frac{ (\zeta-1) \e^\zeta + 1}{ \e^\zeta - 1 } \e^{ - j \zeta /2} \,\mathrm{d}\zeta.
\end{align*}
It is easily verified that
\begin{align*}
\sum_{ j = 1}^\infty s_j = \int_0^\infty \frac{ (\zeta-1) \e^\zeta + 1}{ \e^\zeta - 1 } \frac{1}{ \e^{\zeta/2} - 1 } \mathrm{d}\zeta = c_1'
\end{align*}
and that moreover there is a universal constant $C\in(0,\infty)$ such that
\begin{align*}
\left| \sum_{ j = n+1 }^\infty s_j \right| = \left|  \int_0^\infty \frac{ (\zeta-1) \e^\zeta + 1}{ \e^\zeta - 1 } \frac{\e^{ - n \zeta/2} }{ \e^{\zeta/2} - 1 } \mathrm{d}\zeta  \right| \leq C/n.
\end{align*} 
Again using the fact that $\sum_{ j = 1}^n 1/j = \log n + \gamma + O(1/n)$, the second equation follows. 
\end{proof}

We are almost ready to prove Theorem \ref{thm:gaussian matrix} from its implicit version, Theorem \ref{thm:gaussian implicit}. The final tool in sewing our work together is the following lemma, the proof of which we relegate to the appendix.
\begin{lemma}\label{lem:samstrick}
Let $\sigma, \tilde \sigma >0$ and $\mu, \tilde \mu \in \R$. Assume that $X$ is a random variable such that $\dint_{KS} ((X-\mu)/\sigma,N)\le \varepsilon$, 
where $N$ is a standard Gaussian.
Then it holds
\begin{align*}
\dint_{KS} \left( \frac{ X - \tilde \mu }{ \tilde \sigma }  , N \right) \leq \varepsilon+\frac{|\mu-\tilde \mu|}{\max\{\sigma,\tilde \sigma\}}
+\frac{3}{8} \frac{ |\sigma^2 - \tilde \sigma^2|}{ \min\{\sigma^2,\tilde \sigma^2\}}.
\end{align*}
\end{lemma}

We now prove Theorem \ref{thm:gaussian matrix}.

\begin{proof}[Proof of Theorem \ref{thm:gaussian matrix}]
The proof of Theorem \ref{thm:gaussian matrix} follows immediately from using Theorem \ref{thm:gaussian implicit} and Lemma \ref{lem:gaussian moments} in Lemma \ref{lem:samstrick}. Indeed, in our setting we have $X = \log \mathrm{det}(A^n)$, $\mu = \mathbb{E} [ \log \mathrm{det}(A^n) ]$, $\tilde{\mu} = \frac{1}{2} \log (n-1)! + c_0$, $\sigma = \mathrm{Var}[ \log \mathrm{det}(A^n) ] $, and $\tilde{\sigma} := \frac{1}{2} \log n + c_1$. It follows that with $\varepsilon_n$ and $\delta_n$ as in Lemma \ref{lem:gaussian moments} we have 
\begin{align*}
\mathrm{d}_{KS} \left( \frac{ \log \det (A^n ) - ( \frac{1}{2} \log (n-1)! + c_0 ) }{ \sqrt{ \frac{1}{2} \log n + c_1 } } , G \right)  \leq \frac{C}{ \log^{3/2} n } + \frac{ \varepsilon_n }{ \sqrt{ \frac{1}{2} \log n} } + \frac{ 3}{ 8} \frac{ \delta_n }{ \frac{1}{2} \log n }  \leq \frac{C'}{ \log^{3/2}n }
\end{align*}
for a universal constant $C'\in(0,\infty)$.
\end{proof}

\section{Proofs of Theorems \ref{thm:normalmain}, \ref{thm:stablelimit} and \ref{thm:mixedlimit}}\label{sec:mainproofs}

We use the notation $R_i=\twonorm{Y_i^n}$, $\wt R_i =\log R_i$ and $\bar R_i =\log R_i-\E \log R_i$, as well as
$$\wt \beta_{i/2,j/2}=\log \beta_{i/2,j/2}\,, \quad \betabar_{i/2,j/2}= \log \beta_{i/2,j/2}-\E[\log \beta_{i/2,j/2}].$$
All limits and asymptotic equivalences in this section are for $n\to \infty$ unless stated otherwise.

From \cite[Theorem 3.1]{GKT2019} and its proof we obtain the following lemma.
\begin{lemma}\label{lem:asvariance}
If $(\beta_{i/2,j/2})_{i,j\in \N}$ are independent random variables such that $\beta_{i/2,j/2}$ is $\text{Beta}(i/2,j/2)$ distributed and $p=p_n\to \infty$ is an integer sequence, then it holds
\begin{equation}\label{eq:asvariance}
\Var\left[\frac{1}{2} \sum_{ j = 1}^{p-1} \log  \beta_{ \frac{n-j}{2} , \frac{j}{2} }\right]\sim -\frac{1}{2} \log \frac{n-p+1}{n} -\frac{p^2}{2n(p+1)}\,, \qquad \nto\,.
\end{equation}
\end{lemma}
\begin{proof}
Let $X_1,\ldots,X_{p+1}$ be independent random points in $\R^n$ that are uniformly distributed on the sphere of radius 1 centered at the origin of $\R^n$. Let $\mathcal{V}_{n,p}$ denote the $p$-dimensional volume of the simplex with vertices $X_1,\ldots,X_{p+1}$. Then we have by Theorem 2.5(d) in \cite{GKT2019} that 
\begin{equation}\label{eq:trash}
\xi (1-\xi)^p (p!\mathcal{V}_{n,p})^2 \eid (1-\xi)^p \prod_{ j = 1}^{p-1}   \beta_{ \frac{n-j}{2} , \frac{j}{2} }\,,
\end{equation}
where the random variable $\xi\sim \text{Beta}(n/2,p(n-2)/2)$ is independent of everything else. As in \cite{GKT2019}, we set $\mathcal{L}_{n,p}=\log (p!\mathcal{V}_{n,p})$. Taking logarithm in \eqref{eq:trash} we get
\begin{equation*}
\log \xi +\log (1-\xi)^p +2 \mathcal{L}_{n,p} \eid \log (1-\xi)^p +\sum_{ j = 1}^{p-1}   \log \beta_{ \frac{n-j}{2} , \frac{j}{2} }\,,
\end{equation*}
which implies
\begin{equation*}
\Var\left[\frac{1}{2} \sum_{ j = 1}^{p-1} \log  \beta_{ \frac{n-j}{2} , \frac{j}{2} }\right]=\Var[\mathcal{L}_{n,p}]+\Var[\tfrac{1}{2} \log \xi]\,.
\end{equation*}
From \cite[Theorem 3.1]{GKT2019} we know that 
\begin{equation*}
\Var[\mathcal{L}_{n,p}]\sim -\frac{1}{2} \log \frac{n-p+1}{n} -\frac{p^2}{2n(p+1)}\,, \qquad \nto.
\end{equation*}
Using Lemma~\ref{lem:first three}, we deduce that 
$$\Var[\mathcal{L}_{n,p}]+\Var[\tfrac{1}{2} \log \xi]=\Var[\mathcal{L}_{n,p}] (1+o(1))\,, \qquad \nto,$$ 
which completes the proof of the lemma.
\end{proof}

\subsection{Proof of Theorem \ref{thm:normalmain}}

From \eqref{eq:main rep} and the definition of $b_n$ we get
\begin{align}\label{eq:step1}
&\frac{\log \Vol_p \left( \Delta \Y \right) - b_n}{\sigma_n} \eid \frac{- \log (p!) + \sum_{ i =1}^p \wt R_i + \frac{1}{2} \sum_{ j = 1}^{p-1} \wt \beta_{ \frac{n-j}{2} , \frac{j}{2} }-b_n}{\sigma_n}= T_n+o(1)\,,
\end{align}
where
$$
T_n=\frac{\sum_{ i =1}^p \wt R_i + \frac{1}{2} \sum_{ j = 1}^{p-1} \wt \beta_{ \frac{n-j}{2} , \frac{j}{2} }-p\,\E[\wt R_{(n)} \1_{\{|\wt R_{(n)}|<\sigma_n\}}]-\frac{1}{2} \sum_{j=1}^{p-1}  \E[\betatilde_{ \frac{n-j}{2} , \frac{j}{2} } \1_{\{|\betatilde_{ \frac{n-j}{2} , \frac{j}{2} }|<2\sigma_n\}}]}{\sigma_n}\,. 
$$
Since $\sum_{ i =1}^p \wt R_i + \frac{1}{2} \sum_{ j = 1}^{p-1} \wt \beta_{ \frac{n-j}{2} , \frac{j}{2} }$ is a sum of independent random variables, an application of Theorem~\ref{thm:petrovnormal} shows that  $T_n$ converges in distribution to a standard normal variable $N$, as $\nto$. In conjunction with \eqref{eq:step1}, the desired result \eqref{eq:cltnormal} follows.

\subsection{Proof of Theorem \ref{thm:stablelimit}}

Define 
\begin{equation}\label{eq:cnk}
c_{n}=a_{n} +\int_{-\infty}^{\infty} \frac{x}{1+x^2}\,\mathrm{d}\P(\Rtilde_{(n)}\le \sigma_n(x+a_{n}))\,.
\end{equation}
From \eqref{eq:main rep} and the definition of $b_n$, we get
\begin{align}\label{eq:step3}
\frac{\log \Vol_p \left( \Delta \Y \right) - b_n}{\sigma_n} \eid 
\sigma_n^{-1}\sum_{ i =1}^p (R_i- c_{n}) + \frac{\frac{1}{2} \sum_{ j = 1}^{p-1} \betabar_{ \frac{n-j}{2} , \frac{j}{2} }}{ \sigma_n}+o(1)\,.
\end{align}
We treat the terms on the \rhs~separately. In view of \eqref{eq:dddf} and \eqref{eq:Wspherical}, we have, using the definition of $\betabar$, that
$$\frac{1}{2} \sum_{ j = 1}^{p-1} \betabar_{ \frac{n-j}{2} , \frac{j}{2} }
\eid \omega_n \frac{ \widetilde{W}_{n,p}^{\mathrm{Sph}} \sqrt{\Var[W_{n,p}^{\mathrm{Sph}}]}}{ \omega_n}=: \omega_n \wt Z_{n,p}\,.
$$
Observe that by Lemma \ref{lem:asvariance} we have $\sqrt{\Var[W_{n,p}^{\mathrm{Sph}}]}/\omega_n \to 1$. In combination with Theorem \ref{thm:spherical berry}, we get that $\wt Z_{n,p} \cid N$ as $\nto$ for a standard normal random variable $N$. Using $\omega_n/\sigma_n \to 0$, we conclude that 
$$\frac{\frac{1}{2} \sum_{ j = 1}^{p-1} \betabar_{ \frac{n-j}{2} , \frac{j}{2} }}{ \sigma_n} \eid \frac{\omega_n}{\sigma_n} \wt Z_{n,p} \cip 0\,.$$
By virtue of Slutsky's theorem (see, e.g., \cite{Billingsley1995}) and \eqref{eq:step3}, it remains to show that 
\begin{equation}\label{eq:sdgsd1}
\sigma_n^{-1}\sum_{ i =1}^p (R_i- c_{n}) \cid Z_{\alpha}\,, \qquad \nto\,,
\end{equation}
where the limit random variable $Z_{\alpha}=Z_{\alpha}(c_1,c_2)$ has the
characteristic function \eqref{eq:Zalpha}.
Since $p\to \infty$, condition \eqref{eq:cond1a} is satisfied so that an application of Theorem~\ref{thm:petrovstable} proves 
\eqref{eq:sdgsd1}. The proof of Theorem \ref{thm:stablelimit} is now complete.

\subsection{Proof of Theorem \ref{thm:mixedlimit}}

Recall the notations from the proof of Theorem \ref{thm:stablelimit}. Using $\omega_n/\sigma_n \to q\in (0,\infty)$, we see that 
$$\frac{\frac{1}{2} \sum_{ j = 1}^{p-1} \betabar_{ \frac{n-j}{2} , \frac{j}{2} }}{ \sigma_n} \eid \frac{\omega_n}{\sigma_n} \wt Z_{n,p} \cid q N\,,$$
where $N$ is a standard normal variable independent of $Z_\alpha$.
Following the lines of proof of Theorem \ref{thm:stablelimit}, we conclude that
\begin{align*}
\frac{\log \Vol_p \left( \Delta \Y \right) - b_n}{\sigma_n} \eid 
\sigma_n^{-1}\sum_{ i =1}^p (R_i- c_{n}) + \frac{\frac{1}{2} \sum_{ j = 1}^{p-1} \betabar_{ \frac{n-j}{2} , \frac{j}{2} }}{ \sigma_n}+o(1) \cid q\,Z+Z_{\alpha}\,, \qquad \nto\,,
\end{align*}
finishing the proof of Theorem \ref{thm:mixedlimit}.

\appendix

\section{Some facts about KS-distance}
For random variables $A,B$, define the Kolmogorov-Smirnov distance between $A$ and $B$ by 
\begin{align*}
\dint_{KS}(A,B) := \sup_{ t \in \mathbb{R}} \left| \mathbb{P} \left( A \leq t \right ) - \mathbb{P} \left( B \leq t \right) \right|
\end{align*}
It is straightforward to verify that KS distances satisfy a triangle inequality 
\begin{align} \label{triangle}
\dint_{KS}(A,C) \leq \dint_{KS}(A,B) + \dint_{KS}(B,C)
\end{align}
as well as the fact that KS distances are invariant under affine transformations:
\begin{align} \label{affine}
\dint_{KS}(\gamma A + \lambda, \gamma B + \lambda ) = d ( A,B ).
\end{align}
Finally, if $A,B$ have continuous densities $f$ and $g$, the KS distance is bounded above by the total variation distance:
\begin{align} \label{TV bound}
\dint_{KS}(A,B) \leq d_{TV}(A,B) := \int_{-\infty}^\infty | f(s) - g(s) | \d s.
\end{align}

\begin{proof}[Proof of Lemma \ref{lem:samstrick}]
First, we use the affine invariance \eqref{affine} to obtain first the equality below, and then the triangle inequality \eqref{triangle} to obtain the following inequality, and then the affine invariance with $d ((X-\mu)/\sigma,N)\le \epsilon$ to obtain the final inequality, we have
\begin{align*}
\dint_{KS}\left( \frac{ X - \tilde{\mu} }{ \tilde{\sigma} }  , N \right) &= \dint_{KS}\left( X , \tilde{\sigma} N+ \tilde{\mu} \right)\\
& \leq \dint_{KS}\left( X , \sigma N +\mu \right) + \dint_{KS}\left( \sigma N +\mu  , \tilde{\sigma} N + \tilde{\mu} \right)\\
& \leq \epsilon + \dint_{KS}\left( \sigma N +\mu  , \tilde{\sigma} N + \tilde{\mu} \right)
\end{align*}

It remains to bound the latter term in the final line above. Note that we have 
\begin{align}\label{eq:11111}
 \dint_{KS}\left( \sigma N +\mu  , \tilde{\sigma} N + \tilde{\mu} \right) \leq \dint_{KS}\left( \sigma N + (\mu-\tilde{\mu}) , {\sigma} N \right) + \dint_{KS}\left( \tilde{\sigma} N , \sigma N\right)
\end{align} 
and
\begin{align}\label{eq:11112}
 \dint_{KS}\left( \sigma N +\mu  , \tilde{\sigma} N + \tilde{\mu} \right) \leq \dint_{KS}\left( \tilde \sigma N+ (\tilde{\mu} - \mu ) , \tilde {\sigma} N \right) + \dint_{KS}\left( \tilde{\sigma} N , \sigma N\right)
\end{align} 
The \rhs~ of \eqref{eq:11111} and \eqref{eq:11112} is then bounded by Lemma~\ref{lem:diffmeans} and Lemma \ref{lem:diffvariances}, which completes the proof.
\end{proof}

\begin{lemma}\label{lem:diffvariances}
The KS distance between two Gaussians with mean zero but different variances $\sigma^2>0$ and $\tau^2>0$ is bounded by
\begin{align} \label{varbound}
\sup_{ t \in \mathbb{R}} \left| \int_{ - \infty}^t \frac{1}{ \sqrt{ 2 \pi \sigma^2} } \e^{ - u^2/ 2 \sigma^2} \d u - \int_{ - \infty}^t \frac{1}{ \sqrt{ 2 \pi \tau^2} } \e^{ - u^2/ 2 \tau^2} \d u \right|  \leq \frac{3}{8} \frac{ |\sigma^2 - \tau^2|}{ \min\{\sigma^2,\tau^2\}}\,.
\end{align} 
\end{lemma}

\begin{proof} For a proof of \eqref{varbound} see Lemma 2.5 and Proposition 2.6 of \cite{JP2019} \end{proof}

The KS distance between two Gaussians with the same variance but different means may be bounded as follows.

\begin{lemma}\label{lem:diffmeans}
Let $x > 0$. Then the KS distance between two unit variance Gaussian RVs, one with mean $x$, the other with mean $0$ is bounded above by $x$. That is,
\begin{align*}
d_{KS} \left( \frac{\e^{ - u^2/ 2 } \d u }{ \sqrt{ 2 \pi } } ,  \frac{\e^{ - (u-x)^2/ 2 } \d u }{ \sqrt{ 2 \pi } } \right)  := \sup_{t \in \mathbb{R}} \left| \int_{- \infty}^t  \left( \frac{\e^{ - u^2/ 2 }  }{ \sqrt{ 2 \pi } } -  \frac{\e^{ - (u-x)^2/ 2 } }{ \sqrt{ 2 \pi } } \right) \d u  \right| \leq x.
\end{align*}
\end{lemma}

\begin{proof}
For any $t \in \mathbb{R}$ we have 
\begin{align*}
\sup_{t \in \mathbb{R}}  \left| \int_{- \infty}^t  \left( \frac{\e^{ - u^2/ 2 }  }{ \sqrt{ 2 \pi } } -  \frac{\e^{ - (u-x)^2/ 2 } }{ \sqrt{ 2 \pi } } \right) \d u  \right| &= \sup_{t \in \mathbb{R}} \left| \int_{- \infty}^t  \left( \frac{\e^{ - (u+x/2)^2/ 2 }  }{ \sqrt{ 2 \pi } } -  \frac{\e^{ - (u-x/2)^2/ 2 } }{ \sqrt{ 2 \pi } } \right) \d u  \right|\\ 
&\leq \e^{ - x^2/8} \int_{- \infty}^\infty  \frac{\e^{ - u^2/ 2 } }{ \sqrt{ 2 \pi } } \left| \e^{ux/2} - \e^{ - ux/2} \right| \d u.
\end{align*}
Expanding the power series for $\mathrm{sinh}$ and using the triangle inequality, this is bounded further by
\begin{align*}
\sup_{t \in \mathbb{R}}  \left| \int_{- \infty}^t  \left( \frac{\e^{ - u^2/ 2 }  }{ \sqrt{ 2 \pi } } -  \frac{\e^{ - (u-x)^2/ 2 } }{ \sqrt{ 2 \pi } } \right) \d u  \right| \leq 2 \e^{ - x^2/8} \sum_{ n = 0}^\infty \frac{ (x/2)^{2n+1} }{ (2n+1)! } \mathbb{E}[ |N|^{2n+1} ],
\end{align*}
where $N$ is a standard Gaussian. By Jensen's inequality,
\begin{align*}
\mathbb{E}[ |N|^{2n+1} ] \leq \mathbb{E}[ N^{2n+2} ]^{\frac{2n+1}{2n+2}}.
\end{align*}
Whenever $n \geq 0$, $\mathbb{E}[ N^{2n+2} ] \geq 1$, so we can use the cruder bound with no power adjustment:
\begin{align*}
\mathbb{E}[ |N|^{2n+1} ] \leq \mathbb{E}[ N^{2n+2} ]
\end{align*}
In particular,
\begin{align} \label{apple}
d_{KS} \left( \frac{\e^{ - u^2/ 2 } \d u }{ \sqrt{ 2 \pi } } ,  \frac{\e^{ - (u-x)^2/ 2 } \d u }{ \sqrt{ 2 \pi } } \right) \leq 2 \e^{ - x^2/8} \sum_{ n = 0}^\infty \frac{ (x/2)^{2n+1} }{ (2n+1)! } \mathbb{E}[ N^{2n+2} ].
\end{align}
Finally, the even Gaussian moments are given by 
\begin{align} \label{orange}
\mathbb{E} [ N^{2p+2} ] = \frac{ (2p + 2)! }{ 2^{p+1} (p+1)!} .
\end{align} 
By \eqref{orange}, we have
\begin{align} \label{polar bear}
\sum_{ n = 0}^\infty \frac{ (x/2)^{2n+1} }{ (2n+1)! } \mathbb{E}[ N^{2n+2} ] &= \sum_{ n = 0}^\infty \frac{ (x/2)^{2n+1} }{ (2n+1)! } \frac{ (2n+2)!}{ 2^{n+1} (n+1)!} \nonumber \\ 
&=  \frac{x}{2} \e^{ x^2/8}.
\end{align} 
Plugging \eqref{polar bear} into \eqref{apple}, we obtain the result. 
\end{proof}

\section{Some stable limit theory}

\begin{thm}\cite[Theorem~IV.4.18]{petrov:1975}\label{thm:petrovnormal}
Let $(X_{nk})_{n\ge1, 1\le k\le k_n}$ be a triangular array of real-valued random variables that are independent within rows. Assume there exists a sequence of positive constants $(\sigma_n)_{n\geq 1}$ such that  
\begin{align}
\lim_{\nto} \frac{1}{\sigma_n^2}\sum_{k=1}^{k_n} &\left( \E[X_{nk}^2 \1_{\{|X_{nk}|<\sigma_n\}}]-(\E[X_{nk} \1_{\{|X_{nk}|<\sigma_n\}}])^2 \right) =1\,, \text{ and }\\ 
\lim_{\nto} \sum_{k=1}^{k_n} &\P(|X_{nk}|\ge \vep \sigma_n)=0\,, \qquad  \vep>0\,.
\end{align}
Let $(b_n)_{n\ge 1}$ be a sequence satisfying 
$$b_n=\sum_{k=1}^{k_n}  \E[X_{nk} \1_{\{|X_{nk}|<\sigma_n\}}] +o(\sigma_n)\,, \qquad \nto\,.$$
Then we have
\begin{equation*}
\frac{\sum_{k=1}^{k_n} X_{nk} - b_n}{\sigma_n} \cid Z \sim N(0,1)\,, \qquad \nto\,.
\end{equation*}
\end{thm}

\begin{thm}\label{thm:petrovstable}
Let $(X_{nk})_{n\ge1, 1\le k\le k_n}$ be a triangular array of real-valued random variables that are independent within rows. 
 For some $\alpha\in (0,2)$ and $c_1,c_2\ge 0$ with $c_1+c_2\in(0,\infty)$, assume that there exists a sequence of positive constants $(\sigma_n)_{n\in\N}$ such that, as $\nto$, 
\begin{align}
&\max_{1\le k\le k_n} \P(|X_{nk}|\ge \vep \sigma_n)\to 0\,, \qquad \vep>0\,, \text{ and}\label{eq:cond1a}\\ 
 &\sum_{k=1}^{k_n} \P(\sigma_n^{-1} X_{nk}\le -x)\to c_1 x^{-\alpha}\,,\quad  \sum_{k=1}^{k_n} \P(\sigma_n^{-1}X_{nk}> x)\to c_2 x^{-\alpha}\,, \qquad x>0\,, \text{ and}\\ 
&\lim_{\vep\to 0} \limsup_{\nto} \frac{1}{\sigma_n^2}\sum_{k=1}^{k_n} \left( \E[X_{nk}^2 \1_{\{|X_{nk}|<\vep \sigma_n\}}]-(\E[X_{nk} \1_{\{|X_{nk}|<\vep \sigma_n\}}])^2 \right) =0\,.
\end{align}
For $n\geq 1$ and $1\leq k\leq k_n$, set $a_{nk}:=\sigma_n^{-1}\E[X_{nk} \1_{\{|X_{nk}|<\sigma_n\}}]$ and
let $(b_n)_{n\ge 1}$ be a sequence of real numbers satisfying 
\begin{equation}\label{eq:bnsrw}
b_n=\sum_{k=1}^{k_n} \left(a_{nk} +\int_{-\infty}^{\infty} \frac{x}{1+x^2}\, \mathrm{d}\P(X_{nk}\le \sigma_n(x+a_{nk})) \right) - \gamma +o(1)\,,
\end{equation}
where $\gamma\in \R$.
Then we have the following weak convergence to an $\alpha$-stable limit:
\begin{equation*}
\sigma_n^{-1} \sum_{k=1}^{k_n} X_{nk} - b_n \cid Z_{\alpha} \,, \qquad \nto\,.
\end{equation*}
The limit random variable $Z_{\alpha}=Z_{\alpha}(c_1,c_2,\gamma)$ has the
characteristic function 
\begin{equation}\label{eq:char}
\E[\e^{{\rm i} tZ_{\alpha}}]=\left\{\begin{array}{ll}
\exp \left\{ {\rm i} \gamma t + \alpha (c_1+c_2)\Gamma(-\alpha) \cos (\frac{\pi \alpha}{2}) |t|^\alpha \big(1-{\rm i} \eta \tan (\frac{\pi \alpha}{2}) \sign(t)\big) \right\} \,, 
& \mbox{if } \alpha \neq 1, \\
\exp \left\{ {\rm i} \gamma t - (c_1+c_2)\frac{\pi}{2} |t| \big( 1+{\rm i} \eta \frac{2}{\pi} \sign(t) \log |t| \big) \right\}  \,, 
& \mbox{if } \alpha=1,
\end{array}\right.
\end{equation}
where $\eta=(c_2-c_1)/(c_1+c_2)$.
\end{thm}
\begin{proof}
Without loss of generality we may restrict ourselves to the case $\sigma_n=1$; otherwise replace $X_{nk}$ with $X_{nk}/\sigma_n$. For  $\sigma_n=1$ and noting that \eqref{eq:cond1a} is Petrov's so-called infinite smallness condition, \cite[Theorem IV.2.8]{petrov:1975} yields the existence of a sequence of constants $b_n$ such that $\sum_{k=1}^{k_n} X_{nk} - b_n$ converges in distribution to an infinitely divisible random variable $Z_{\alpha}$ with L\'evy spectral function $L(x)=c_1 |x|^{-\alpha} \1_{\{x<0\}} -c_2 x^{-\alpha}\1_{\{x>0\}}$ for $x\in\R$. By \cite[Theorem IV.2.5]{petrov:1975}, $b_n$ may be chosen as in \eqref{eq:bnsrw}. 
From the form of $L(x)$ we can deduce by \cite[Theorem IV.3.11]{petrov:1975} that the limit variable $Z_{\alpha}$ has a stable distribution with characteristic function 
\begin{equation*}
\exp \left\{ {\rm i} \gamma t- \int_{-\infty}^{\infty} \Big(\e^{{\rm i} tx}-1-\frac{{\rm i} tx}{1+x^2}\Big) \,\mathrm{d}L(x) \right\}\,,
\end{equation*}
where $L$ is the L\'evy spectral function from above.
Finally, by parts (i) and (iv) of \cite[Theorem 3.3]{janson:2011} (with $c_2=c_+$ and $c_1=c_-$), this expression equals the \rhs~in \eqref{eq:char}. We mention that an alternative proof of the last step can be furnished by using \cite[Theorem IV.3.12]{petrov:1975}.
\end{proof}

\bibliographystyle{plain}
\bibliography{limit_thm}

\end{document}